\documentclass{article}

\usepackage[nonatbib,final]{neurips_2021}

\usepackage[utf8]{inputenc} % allow utf-8 input
\usepackage[T1]{fontenc}    % use 8-bit T1 fonts
\usepackage{amssymb,amsmath,amsthm,amsfonts}
\usepackage{bm}
\usepackage{textgreek}
\usepackage{mathtools}
\usepackage{enumitem}
\usepackage[font=small]{caption}
\usepackage[labelformat=simple]{subcaption}
\usepackage{float}
\usepackage[ruled,vlined,linesnumbered]{algorithm2e}
\usepackage{physics}
\usepackage{footnote}
\usepackage{xcolor}
\usepackage{mathrsfs}
\usepackage{bbm}
\usepackage[colorlinks]{hyperref}
\usepackage{cleveref}
\usepackage{thm-restate}

\newtheorem{theorem}{Theorem}
\newtheorem{definition}[theorem]{Definition}
\newtheorem{lemma}[theorem]{Lemma}
\newtheorem{proposition}[theorem]{Proposition}

\newtheorem{remark}[theorem]{Remark}
\newtheorem{assumption}{Assumption}

\newcommand{\eqn}[1]{(\ref{eqn:#1})}

\newcommand{\thm}[1]{\hyperref[thm:#1]{Theorem~\ref*{thm:#1}}}
\newcommand{\cor}[1]{\hyperref[cor:#1]{Corollary~\ref*{cor:#1}}}
\newcommand{\defn}[1]{\hyperref[defn:#1]{Definition~\ref*{defn:#1}}}
\newcommand{\lem}[1]{\hyperref[lem:#1]{Lemma~\ref*{lem:#1}}}
\newcommand{\prop}[1]{\hyperref[prop:#1]{Proposition~\ref*{prop:#1}}}
\newcommand{\assum}[1]{\hyperref[assum:#1]{Assumption~\ref*{assum:#1}}}
\newcommand{\fig}[1]{\hyperref[fig:#1]{Figure~\ref*{fig:#1}}}
\newcommand{\tab}[1]{\hyperref[tab:#1]{Table~\ref*{tab:#1}}}
\newcommand{\algo}[1]{\hyperref[algo:#1]{Algorithm~\ref*{algo:#1}}}
\renewcommand{\sec}[1]{\hyperref[sec:#1]{Section~\ref*{sec:#1}}}
\newcommand{\append}[1]{\hyperref[append:#1]{Appendix~\ref*{append:#1}}}
\newcommand{\fac}[1]{\hyperref[fac:#1]{Fact~\ref*{fac:#1}}}
\newcommand{\lin}[1]{\hyperref[lin:#1]{Line~\ref*{lin:#1}}}
\newcommand{\fnote}[1]{\hyperref[fnote:#1]{Footnote~\ref*{fnote:#1}}}

\def\>{\rangle}
\def\<{\langle}

\newcommand{\vect}[1]{\ensuremath{\mathbf{#1}}}
\newcommand{\x}{\ensuremath{\mathbf{x}}}

\newcommand{\R}{\mathbb{R}}

\DeclareMathOperator{\poly}{poly}

\renewcommand{\d}{\mathrm{d}}

\title{Escape saddle points by a simple gradient-descent based algorithm}

\author{
  \textbf{Chenyi Zhang}$^{1}$\qquad \textbf{Tongyang Li}$^{2,3,4}$\thanks{Corresponding author. Email: tongyangli@pku.edu.cn} \\
  $^1$ Institute for Interdisciplinary Information Sciences, Tsinghua University, China \\
  $^2$ Center on Frontiers of Computing Studies, Peking University, China \\
  $^3$ School of Computer Science, Peking University, China \\
  $^4$ Center for Theoretical Physics, Massachusetts Institute of Technology, USA
}

\begin{document}

\maketitle

\begin{abstract}
Escaping saddle points is a central research topic in nonconvex optimization. In this paper, we propose a simple gradient-based algorithm such that for a smooth function $f\colon\mathbb{R}^n\to\mathbb{R}$, it outputs an $\epsilon$-approximate second-order stationary point in $\tilde{O}(\log n/\epsilon^{1.75})$ iterations. Compared to the previous state-of-the-art algorithms by Jin et al. with $\tilde{O}(\log^4 n/\epsilon^{2})$ or $\tilde{O}(\log^6 n/\epsilon^{1.75})$ iterations, our algorithm is polynomially better in terms of $\log n$ and matches their complexities in terms of $1/\epsilon$. For the stochastic setting, our algorithm outputs an $\epsilon$-approximate second-order stationary point in $\tilde{O}(\log^{2} n/\epsilon^{4})$ iterations. Technically, our main contribution is an idea of implementing a robust Hessian power method using only gradients, which can find negative curvature near saddle points and achieve the polynomial speedup in $\log n$ compared to the perturbed gradient descent methods. Finally, we also perform numerical experiments that support our results.
 \end{abstract}

%%%%%%%%%%%%%%%%%%%%%%%%%%%%%%%%%%%%%%%%%%%%%%%%%%%%%%%%%%%%%%%%%

\setcounter{footnote}{0} 

\section{Introduction}
Nonconvex optimization is a central research area in optimization theory, since lots of modern machine learning problems can be formulated in models with nonconvex loss functions, including deep neural networks, principal component analysis, tensor decomposition, etc. In general, finding a global minimum of a nonconvex function is NP-hard in the worst case. Instead, many theoretical works focus on finding a local minimum instead of a global one, because recent works (both empirical and theoretical) suggested that local minima are nearly as good as global minima for a significant amount of well-studied machine learning problems; see e.g.~\cite{bhojanapalli2016global,ge2015escaping,ge2016matrix,ge2018learning,ge2017optimization,hardt2018gradient}. On the other hand, saddle points are major obstacles for solving these problems, not only because they are ubiquitous in high-dimensional settings where the directions for escaping may be few (see e.g.~\cite{bray2007statistics,dauphin2014identifying,fyodorov2007replica}), but also saddle points can correspond to highly suboptimal solutions (see e.g.~\cite{jain2017global,sun2018geometric}).

Hence, one of the most important topics in nonconvex optimization is to \emph{escape saddle points}. Specifically, we consider a twice-differentiable function $f\colon\mathbb{R}^n\to\mathbb{R}$ such that
\vspace{-1mm}
\begin{itemize}
\item $f$ is $\ell$-smooth: $\|\nabla f(\vect{x}_1)-\nabla f(\vect{x}_2)\|\leq\ell\|\vect{x}_1-\vect{x}_2\|\quad\forall \vect{x}_1,\vect{x}_2\in\mathbb{R}^n$,
    \vspace{-0.5mm}
\item $f$ is $\rho$-Hessian Lipschitz: $\|\mathcal{H}(\vect{x}_1)-\mathcal{H}(\vect{x}_2)\|\leq\rho\|\vect{x}_1-\vect{x}_2\|\quad\forall\vect{x}_1,\vect{x}_2\in\mathbb{R}^n$;
\end{itemize}
\vspace{-1mm}
here $\mathcal{H}$ is the Hessian of $f$. The goal is to find an $\epsilon$-approximate second-order stationary point $\vect{x}_{\epsilon}$:\footnote{We can ask for an $(\epsilon_{1},\epsilon_{2})$-approx.~second-order stationary point s.t. $\|\nabla f(\x)\|\leq\epsilon_{1}$ and $\lambda_{\min}(\nabla^{2}f(\x))\geq-\epsilon_{2}$ in general. The scaling in \eqn{eps-approx-local-min} was adopted as a standard in literature~\cite{agarwal2017finding,carmon2018accelerated,fang2019sharp,jin2017escape,jin2019stochastic,jin2018accelerated,nesterov2006cubic,tripuraneni2018stochastic,xu2017neon,xu2018first}.}
\begin{align}\label{eqn:eps-approx-local-min}
\|\nabla f(\vect{x}_{\epsilon})\|\leq\epsilon,\quad\lambda_{\min}(\mathcal{H}(\vect{x}_{\epsilon}))\geq-\sqrt{\rho\epsilon}.
\end{align}
In other words, at any $\epsilon$-approx. second-order stationary point $\x_{\epsilon}$, the gradient is small with norm being at most $\epsilon$ and the Hessian is close to be positive semi-definite with all its eigenvalues $\geq-\sqrt{\rho\epsilon}$.

Algorithms for escaping saddle points are mainly evaluated from two aspects. On the one hand, considering the enormous dimensions of machine learning models in practice, dimension-free or almost dimension-free (i.e., having $\poly(\log n)$ dependence) algorithms are highly preferred. On the other hand, recent empirical discoveries in machine learning suggests that it is often feasible to tackle difficult real-world problems using simple algorithms, which can be implemented and maintained more easily in practice. On the contrary, algorithms with nested loops often suffer from significant overheads in large scales, or introduce concerns with the setting of hyperparameters and numerical stability (see e.g.~\cite{agarwal2017finding,carmon2018accelerated}), making them relatively hard to find practical implementations. 

It is then natural to explore simple gradient-based algorithms for escaping from saddle points. The reason we do not assume access to Hessians is because its construction takes $\Omega(n^2)$ cost in general, which is computationally infeasible when the dimension is large. A seminal work along this line was by Ge et al.~\cite{ge2015escaping}, which found an $\epsilon$-approximate second-order stationary point satisfying \eqn{eps-approx-local-min} using only gradients in $O(\poly(n,1/\epsilon))$ iterations. This is later improved to be almost dimension-free $\tilde{O}(\log^{4} n/\epsilon^{2})$ in the follow-up work~\cite{jin2017escape},\footnote{The $\tilde{O}$ notation omits poly-logarithmic terms, i.e., $\tilde{O}(g)=O(g\poly(\log g))$.} and the perturbed accelerated gradient descent algorithm~\cite{jin2018accelerated} based on Nesterov's accelerated gradient descent~\cite{nesterov1983method} takes $\tilde{O}(\log^{6} n/\epsilon^{1.75})$ iterations. However, these results still suffer from a significant overhead in terms of $\log n$. On the other direction, Refs.~\cite{allen2018neon2,liu2018adaptive,xu2017neon} demonstrate that an $\epsilon$-approximate second-order stationary point can be find using gradients in $\tilde{O}(\log n/\epsilon^{1.75})$ iterations. Their results are based on previous works \cite{agarwal2017finding,carmon2018accelerated} using Hessian-vector products and the observation that the Hessian-vector product can be approximated via the difference of two gradient queries. Hence, their implementations contain nested-loop structures with relatively large numbers of hyperparameters. It has been an open question whether it is possible to keep both the merits of using only first-order information as well as being close to dimension-free using a simple, gradient-based algorithm without a nested-loop structure~\cite{jordan2017talk}. This paper answers this question in the affirmative.

%================================================================

\paragraph{Contributions.}
Our main contribution is a simple, single-loop, and robust gradient-based algorithm that can find an $\epsilon$-approximate second-order stationary point of a smooth, Hessian Lipschitz function $f\colon\mathbb{R}^n\to\mathbb{R}$. Compared to previous works \cite{allen2018neon2,liu2018adaptive,xu2017neon} exploiting the idea of gradient-based Hessian power method, our algorithm has a single-looped, simpler structure and better numerical stability. Compared to the previous state-of-the-art results with single-looped structures by~\cite{jin2018accelerated} and~\cite{jin2017escape,jin2019stochastic} using $\tilde{O}(\log^{6}n/\epsilon^{1.75})$ or $\tilde{O}(\log^{4}n/\epsilon^{2})$ iterations, our algorithm achieves a polynomial speedup in $\log n$: \begin{theorem}[informal]\label{thm:main-intro}
Our single-looped algorithm finds an $\epsilon$-approximate second-order stationary point in $\tilde{O}(\log n/\epsilon^{1.75})$ iterations.
\end{theorem}
Technically, our work is inspired by the perturbed gradient descent (PGD) algorithm in~\cite{jin2017escape,jin2019stochastic} and the perturbed accelerated gradient descent (PAGD) algorithm in~\cite{jin2018accelerated}. Specifically, PGD applies gradient descents iteratively until it reaches a point with small gradient, which can be a potential saddle point. Then PGD generates a uniform perturbation in a small ball centered at that point and then continues the GD procedure. It is demonstrated that, with an appropriate choice of the perturbation radius, PGD can shake the point off from the neighborhood of the saddle point and converge to a second-order stationary point with high probability. The PAGD in~\cite{jin2018accelerated} adopts a similar perturbation idea, but the GD is replaced by Nesterov's AGD~\cite{nesterov1983method}.

Our algorithm is built upon PGD and PAGD but with one main modification regarding the perturbation idea: it is more efficient to add a perturbation in the \emph{negative curvature direction} nearby the saddle point, rather than the uniform perturbation in PGD and PAGD, which is a compromise since we generally cannot access the Hessian at the saddle due to its high computational cost. Our key observation lies in the fact that we do not have to compute the entire Hessian to detect the negative curvature. Instead, in a small neighborhood of a saddle point, gradients can be viewed as Hessian-vector products plus some bounded deviation. In particular, GD near the saddle with learning rate $1/\ell$ is approximately the same as the power method of the matrix $(I-\mathcal{H}/\ell)$. As a result, the most negative eigenvalues stand out in GD because they have leading exponents in the power method, and thus it approximately moves along the direction of the most negative curvature nearby the saddle point. Following this approach, we can escape the saddle points more rapidly than previous algorithms: for a constant $\epsilon$, PGD and PAGD take $O(\log n)$ iterations to decrease the function value by $\Omega(1/\log^3 n)$ and $\Omega(1/\log^5 n)$ with high probability, respectively; on the contrary, we can first take $O(\log n)$ iterations to specify a negative curvature direction, and then add a larger perturbation in this direction to decrease the function value by $\Omega(1)$. See \prop{NC-finding} and \prop{ANC-finding}. After escaping the saddle point, similar to PGD and PAGD, we switch back to GD and AGD iterations, which are efficient to decrease the function value when the gradient is large~\cite{jin2017escape,jin2019stochastic,jin2018accelerated}. 

Our algorithm is also applicable to the stochastic setting where we can only access stochastic gradients, and the stochasticity is not under the control of our algorithm. We further assume that the stochastic gradients are Lipschitz (or equivalently, the underlying functions are gradient-Lipschitz, see \assum{stochastic-lipschitz}), which is also adopted in most of the existing works; see e.g.~\cite{fang2018spider,jin2017escape,jin2019stochastic,zhou2018finding}. We demonstrate that a simple extended version of our algorithm takes $O(\log^2 n)$ iterations to detect a negative curvature direction using only stochastic gradients, and then obtain an $\Omega(1)$ function value decrease with high probability. On the contrary, the perturbed stochastic gradient descent (PSGD) algorithm in~\cite{jin2017escape,jin2019stochastic}, the stochastic version of PGD, takes $O(\log^{10}n)$ iterations to decrease the function value by $\Omega(1/\log^{5}n)$ with high probability.

\begin{theorem}[informal]\label{thm:main-intro-SGD}
In the stochastic setting, our algorithm finds an $\epsilon$-approximate second-order stationary point using $\tilde{O}(\log^2 n/\epsilon^{4})$ iterations via stochastic gradients.
\end{theorem}

Our results are summarized in \tab{main}. Although the underlying dynamics in~\cite{allen2018neon2,liu2018adaptive,xu2017neon} and our algorithm have similarity, the main focus of our work is different. Specifically, Refs.~\cite{allen2018neon2,liu2018adaptive,xu2017neon} mainly aim at using novel techniques to reduce the iteration complexity for finding a second-order stationary point, whereas our work mainly focuses on reducing the number of loops and hyper-parameters of negative curvature finding methods while preserving their advantage in iteration complexity, since a much simpler structure accords with empirical observations and enables wider applications. Moreover, the choice of perturbation in \cite{allen2018neon2} is based on the Chebyshev approximation theory, which may require additional nested-looped structures to boost the success probability.
In the stochastic setting, there are also other results studying nonconvex optimization~\cite{ge2019stabilized,li2019ssrgd,yu2018third,zhou2020stochastic,ge2017no,zhang2018primal,zhou2018stochastic} from different perspectives than escaping saddle points, which are incomparable to our results.

\begin{table}[ht]
\centering
\begin{tabular}{ccccc}
\hline
Setting & Reference & Oracle & Iterations & Simplicity \\ \hline
Non-stochastic & \cite{agarwal2017finding,carmon2018accelerated} & Hessian-vector product & $\tilde{O}(\log n/\epsilon^{1.75})$ & Nested-loop \\ \hline
Non-stochastic & \cite{jin2017escape,jin2019stochastic} & Gradient & $\tilde{O}(\log^{4} n/\epsilon^{2})$ & Single-loop \\ \hline
Non-stochastic & \cite{jin2018accelerated} & Gradient & $\tilde{O}(\log^{6} n/\epsilon^{1.75})$  & Single-loop \\ \hline
Non-stochastic & \cite{allen2018neon2,liu2018adaptive,xu2017neon} & Gradient & $\tilde{O}(\log n/\epsilon^{1.75})$  & Nested-loop \\ \hline
Non-stochastic & \textbf{this work} & Gradient & $\tilde{O}(\log n/\epsilon^{1.75})$  & Single-loop \\
\hline
Stochastic & \cite{jin2017escape,jin2019stochastic} & Gradient & $\tilde{O}(\log^{15} n/\epsilon^{4})$  & Single-loop \\ \hline
Stochastic & \cite{fang2019sharp} & Gradient & $\tilde{O}(\log^5 n/\epsilon^{3.5})$ & Single-loop \\ \hline
Stochastic & \cite{allen2018neon2} & Gradient & $\tilde{O}(\log^2 n/\epsilon^{3.5})$ & Nested-loop \\ \hline
Stochastic & \cite{fang2018spider} & Gradient & $\tilde{O}(\log^2 n/\epsilon^{3})$ & Nested-loop \\ \hline
Stochastic & \textbf{this work} & Gradient & $\tilde{O}(\log^2 n/\epsilon^{4})$ & Single-loop \\
\hline
\end{tabular}
\vspace{2mm}
\caption{A summary of the state-of-the-art results on finding approximate second-order stationary points by the first-order (gradient) oracle. Iteration numbers are highlighted in terms of the dimension $n$ and the precision $\epsilon$.}
\label{tab:main}
\end{table}

It is worth highlighting that our gradient-descent based algorithm enjoys the following nice features:
\begin{itemize}[leftmargin=*]
\item \emph{Simplicity:} Some of the previous algorithms have nested-loop structures with the concern of practical impact when setting the hyperparameters. In contrast, our algorithm based on negative curvature finding only contains a single loop with two components: gradient descent (including AGD or SGD) and perturbation. As mentioned above, such simple structure is preferred in machine learning, which increases the possibility of our algorithm to find real-world applications.

\item\emph{Numerical stability:} Our algorithm contains an additional renormalization step at each iteration when escaping from saddle points. Although in theoretical aspect a renormalization step does not affect the output and the complexity of our algorithm, when finding negative curvature near saddle points it enables us to sample gradients in a larger region, which makes our algorithm more numerically stable against floating point error and other errors. The introduction of renormalization step is enabled by the simple structure of our algorithm, which may not be feasible for more complicated algorithms~\cite{allen2018neon2,liu2018adaptive,xu2017neon}.

\item \emph{Robustness:} Our algorithm is robust against adversarial attacks when evaluating the value of the gradient. Specifically, when analyzing the performance of our algorithm near saddle points, we essentially view the deflation from pure quadratic geometry as an external noise. Hence, the effectiveness of our algorithm is unaffected under external attacks as long as the adversary is bounded by deflations from quadratic landscape.
\end{itemize}

Finally, we perform numerical experiments that support our polynomial speedup in $\log n$. We perform our negative curvature finding algorithms using GD or SGD in various landscapes and general classes of nonconvex functions, and use comparative studies to show that our \algo{NC-finding} and \algo{SNC-finding} achieve a higher probability of escaping saddle points using much fewer iterations than PGD and PSGD (typically less than $1/3$ times of the iteration number of PGD and $1/2$ times of the iteration number of PSGD, respectively). Moreover, we perform numerical experiments benchmarking the solution quality and iteration complexity of our algorithm against accelerated methods. Compared to PAGD \cite{jin2018accelerated} and even advanced optimization algorithms such as NEON+~\cite{xu2017neon}, \algo{ANCGD} possesses better solution quality and iteration complexity in various landscapes given by more general nonconvex functions. With fewer iterations compared to PAGD and NEON+ (typically less than $1/3$ times of the iteration number of PAGD and $1/2$ times of the iteration number of NEON+, respectively), our \algo{ANCGD} achieves a higher probability of escaping from saddle points.

\paragraph{Open questions.} This work leaves a couple of natural open questions for future investigation:
\begin{itemize}[leftmargin=*]
\item Can we achieve the polynomial speedup in $\log n$ for more advanced stochastic optimization algorithms with complexity $\tilde{O}(\poly(\log n)/\epsilon^{3.5})$~\cite{allen2018natasha2,allen2018neon2,fang2019sharp,tripuraneni2018stochastic,xu2018first} or $\tilde{O}(\poly(\log n)/\epsilon^{3})$~\cite{fang2018spider,zhou2019stochastic}?
\item How is the performance of our algorithms for escaping saddle points in real-world applications, such as tensor decomposition~\cite{ge2015escaping,ge2017optimization}, matrix completion~\cite{ge2016matrix}, etc.?
\end{itemize}

%==========================================================
\paragraph{Broader impact.} This work focuses on the theory of nonconvex optimization, and as far as we see, we do not anticipate its potential negative societal impact. Nevertheless, it might have a positive impact for researchers who are interested in understanding the theoretical underpinnings of (stochastic) gradient descent methods for machine learning applications.

%==========================================================

\paragraph{Organization.}
In \sec{SNC}, we introduce our gradient-based Hessian power method algorithm for negative curvature finding, and present how our algorithms provide polynomial speedup in $\log n$ for both PGD and PAGD. In \sec{SNC-setting}, we present the stochastic version of our negative curvature finding algorithm using stochastic gradients and demonstrate its polynomial speedup in $\log n$ for PSGD. Numerical experiments are presented in \sec{numerical}. We provide detailed proofs and additional numerical experiments in the supplementary material.

%%%%%%%%%%%%%%%%%%%%%%%%%%%%%%%%%%%%%%%%%%%%%%%%%%%%%%%%%%%%%%%%%%%%%%%%%%%%%%

\section{A simple algorithm for negative curvature finding}\label{sec:SNC}
We show how to find negative curvature near a saddle point using a gradient-based Hessian power method algorithm, and extend it to a version with faster convergence rate by replacing gradient descents by accelerated gradient descents. The intuition works as follows: in a small enough region nearby a saddle point, the gradient can be approximately expressed as a Hessian-vector product formula, and the approximation error can be efficiently upper bounded, see Eq.~\eqn{gradient-approximation-0}. Hence, using only gradients information, we can implement an accurate enough Hessian power method to find negative eigenvectors of the Hessian matrix, and further find the negative curvature nearby the saddle.
\subsection{Negative curvature finding based on gradient descents}\label{sec:GD+NC}
We first present an algorithm for negative curvature finding based on gradient descents. Specifically, for any $\tilde{\vect{x}}\in\mathbb{R}^{n}$ with $\lambda_{\min}(\mathcal{H}(\tilde{\vect{x}}))\leq-\sqrt{\rho\epsilon}$, it finds a unit vector $\hat{\vect{e}}$ such that    $\hat{\vect{e}}^{T}\mathcal{H}(\tilde{\vect{x}})\hat{\vect{e}}\leq-\sqrt{\rho\epsilon}/4$.
\begin{algorithm}[htbp]
\caption{Negative Curvature Finding($\tilde{\vect{x}},r,\mathscr{T}$).}
\label{algo:NC-finding}
$\vect{y}_0\leftarrow$Uniform$(\mathbb{B}_
{\tilde{\vect{x}}}(r))$ where $\mathbb{B}_{\tilde{\vect{x}}}(r)$ is the $\ell_{2}$-norm ball centered at $\tilde{\vect{x}}$ with radius $r$\;
\For{$t=1,...,\mathscr{T}$}{
$\vect{y}_{t}\leftarrow \vect{y}_{t-1}-\frac{\|\vect{y}_{t-1}\|}{\ell r}\big(\nabla f(\tilde{\vect{x}}+r\vect{y}_{t-1}/\|\vect{y}_{t-1}\|)-\nabla f(\tilde{\vect{x}})\big)$ \label{lin:y_t-definition}\;
}
\textbf{Output} $\vect{y}_{\mathscr{T}}/r$.
\end{algorithm}

\begin{restatable}{proposition}{NC-finding}\label{prop:NC-finding}
Suppose the function $f\colon\mathbb{R}^n\to\mathbb{R}$ is $\ell$-smooth and $\rho$-Hessian Lipschitz. For any $0<\delta_0\leq1$, we specify our choice of parameters and constants we use as follows:
\begin{align}\label{eqn:parameter-choice}
\mathscr{T}=\frac{8\ell}{\sqrt{\rho\epsilon}}\cdot\log\Big(\frac{\ell}{\delta_0}\sqrt{\frac{n}{\pi\rho\epsilon}}\Big),\quad r=\frac{\epsilon}{8\ell}\sqrt{\frac{\pi}{n}}\delta_0.
\end{align}
Suppose that $\tilde{\vect{x}}$ satisfies $\lambda_{\min}(\nabla^2f(\tilde{\vect{x}}))\leq-\sqrt{\rho\epsilon}$. Then with probability at least $1-\delta_0$, \algo{NC-finding} outputs a unit vector $\hat{\vect{e}}$ satisfying
\begin{align}
\hat{\vect{e}}^{T}\mathcal{H}(\vect{x})\hat{\vect{e}}\leq-\sqrt{\rho\epsilon}/4,
\end{align}
using $O(\mathscr{T})=\tilde{O}\Big(\frac{\log n}{\sqrt{\rho\epsilon}}\Big)$ iterations, where $\mathcal{H}$ stands for the Hessian matrix of function $f$.
\end{restatable}
\begin{proof}
Without loss of generality we assume $\tilde{\vect{x}}=\vect{0}$ by shifting $\R^{n}$ such that $\tilde{\vect{x}}$ is mapped to $\vect{0}$. Define a new $n$-dimensional function
\begin{align}
h_f(\vect{x}):=f(\vect{x})-\left\<\nabla f(\vect{0}),\vect{x}\right\>,
\end{align}
for the ease of our analysis. Since $\left\<\nabla f(\vect{0}),\vect{x}\right\>$ is a linear function with Hessian being 0, the Hessian of $h_f$ equals to the Hessian of $f$, and $h_f(\vect{x})$ is also $\ell$-smooth and $\rho$-Hessian Lipschitz. In addition, note that $\nabla h_f(\vect{0})=\nabla f(\vect{0})-\nabla f(\vect{0})=0$. Then for all $\vect{x}\in\mathbb{R}^n$,
\begin{align}
\nabla h_f(\vect{x})=\int_{\xi=0}^{1}\mathcal{H}(\xi\vect{x})\cdot\vect{x}\,\d \xi=\mathcal{H}(\vect{0})\vect{x}+\int_{\xi=0}^{1}(\mathcal{H}(\xi\vect{x})-\mathcal{H}(\vect{0}))\cdot \vect{x}\,\d\xi.
\end{align}
Furthermore, due to the $\rho$-Hessian Lipschitz condition of both $f$ and $h_f$, for any $\xi\in[0,1]$ we have $\|\mathcal{H}(\xi\vect{x})-\mathcal{H}(\vect{0})\|\leq\rho\|\vect{x}\|$, which leads to
\begin{align}\label{eqn:gradient-approximation-0}
\|\nabla h_f(\vect{x})-\mathcal{H}(\vect{0})\vect{x}\|\leq\rho\|\vect{x}\|^2.
\end{align}

Observe that the Hessian matrix $\mathcal{H}(\vect{0})$ admits the following eigen-decomposition:
\begin{align}\label{eqn:GD-Hessian-decomposition}
\mathcal{H}(\vect{0})=\sum_{i=1}^{n}\lambda_i\vect{u}_{i}\vect{u}_i^{T},
\end{align}
where the set $\{\vect{u}_i\}_{i=1}^{n}$ forms an orthonormal basis of $\mathbb{R}^n$. Without loss of generality, we assume the eigenvalues $\lambda_1,\lambda_2,\ldots,\lambda_n$ corresponding to $\vect{u}_1,\vect{u}_2,\ldots,\vect{u}_n$ satisfy
\begin{align}
\lambda_1\leq\lambda_2\leq\cdots\leq\lambda_n,
\end{align}
in which $\lambda_1\leq-\sqrt{\rho\epsilon}$. If $\lambda_n\leq-\sqrt{\rho\epsilon}/2$, \prop{NC-finding} holds directly. Hence, we only need to prove the case where $\lambda_n>-\sqrt{\rho\epsilon}/2$, in which there exists some $p>1$, $p'>1$ with
\begin{align}
\lambda_p\leq -\sqrt{\rho\epsilon}\leq \lambda_{p+1},\quad\lambda_{p'}\leq -\sqrt{\rho\epsilon}/2< \lambda_{p'+1}.
\end{align}
We use $\mathfrak{S}_{\parallel}$, $\mathfrak{S}_{\perp}$ to separately denote the subspace of $\mathbb{R}^{n}$ spanned by $\left\{\vect{u}_1,\vect{u}_2,\ldots,\vect{u}_p\right\}$, $\left\{\vect{u}_{p+1},\vect{u}_{p+2},\ldots,\vect{u}_{n}\right\}$, and use $\mathfrak{S}_{\parallel}'$, $\mathfrak{S}_{\perp}'$ to denote the subspace of $\mathbb{R}^{n}$ spanned by $\left\{\vect{u}_1,\vect{u}_2,\ldots,\vect{u}_{p'}\right\}$, $\left\{\vect{u}_{p'+1},\vect{u}_{p+2},\ldots,\vect{u}_{n}\right\}$. Furthermore, we define
\begin{align}\label{eqn:para-perp-defn}
\vect{y}_{t,\parallel}&:=\sum_{i=1}^p\left\<\vect{u}_i,\vect{y}_t\right\>\vect{u}_i, \qquad\ 
\vect{y}_{t,\perp}:=\sum_{i=p}^n\left\<\vect{u}_i,\vect{y}_t\right\>\vect{u}_i; \\
\vect{y}_{t,\parallel'}&:=\sum_{i=1}^{p'}\left\<\vect{u}_i,\vect{y}_t\right\>\vect{u}_i, \qquad
\vect{y}_{t,\perp'}:=\sum_{i=p'}^n\left\<\vect{u}_i,\vect{y}_t\right\>\vect{u}_i
\end{align}
respectively to denote the component of $\vect{y}_t$ in \lin{y_t-definition} in the subspaces $\mathfrak{S}_{\parallel}$, $\mathfrak{S}_{\perp}$, $\mathfrak{S}_{\parallel}'$, $\mathfrak{S}_{\perp}'$, and let $\alpha_{t}:=\|\vect{y}_{t,\parallel}\|/\|\vect{y}_{t}\|$. Observe that
\begin{align}
\Pr\left\{\alpha_0\geq\delta_0\sqrt{\pi/n}\right\}\geq\Pr\left\{|y_{0,1}|/r\geq\delta_0\sqrt{\pi/n}\right\},
\end{align}
where $y_{0,1}:=\left\<\vect{u}_1,\vect{y}_0\right\>$ denotes the component of $\vect{y}_0$ along $\vect{u}_1$. Consider the case where $\alpha_{0}\geq\delta_0\sqrt{\pi/n}$, which can be achieved with probability  \begin{align}\label{eqn:alpha0-probability}
\Pr\left\{\alpha_0\geq\sqrt{\frac{\pi}{n}}\delta_0\right\}\geq1-\sqrt{\frac{\pi}{n}}\delta_0\cdot\frac{\text{Vol}(\mathbb{B}_0^{n-1}(1))}{\text{Vol}(\mathbb{B}_0^{n}(1))}\geq 1-\sqrt{\frac{\pi}{n}}\delta_0\cdot\sqrt{\frac{n}{\pi}}=1-\delta_0.
\end{align}
We prove that there exists some $t_0$ with $1\leq t_0\leq \mathscr{T}$ such that
\begin{align}
\|\vect{y}_{t_0,\perp'}\|/\|\vect{y}_{t_0}\|\leq\sqrt{\rho\epsilon}/(8\ell).
\end{align}
Assume the contrary, for any $1\leq t\leq\mathscr{T}$, we all have $\|\vect{y}_{t,\perp'}\|/\|\vect{y}_{t}\|>\sqrt{\rho\epsilon}/(8\ell)$. Then $\|\vect{y}_{t,\perp}'\|$ satisfies the following recurrence formula:
\begin{align}
\|\vect{y}_{t+1,\perp'}\|\leq (1+\sqrt{\rho\epsilon}/(2\ell))\|\vect{y}_{t,\perp'}\|+\|\Delta_{\perp'}\|\leq
(1+\sqrt{\rho\epsilon}/(2\ell)+\|\Delta\|/\|\vect{y}_{t,\perp'}\|)\|\vect{y}_{t,\perp'}\|,
\end{align}
where 
\begin{align}\label{eqn:Delta-defn}
\Delta:=\frac{\|\vect{y}_t\|}{r\ell}(\nabla h_f(r\vect{y}_t/\|\vect{y}_t\|)-\mathcal{H}(0)\cdot (r\vect{y}_t/\|\vect{y}_t\|))
\end{align}
stands for the deviation from pure quadratic approximation and $\|\Delta\|/\|\vect{y}_t\|\leq\rho r/\ell$ due to \eqn{gradient-approximation-0}. Hence,
\begin{align}
\|\vect{y}_{t+1,\perp'}\|\leq\Big(1+\frac{\sqrt{\rho\epsilon}}{2\ell}+\frac{\|\Delta\|}{\|\vect{y}_{t,\perp'}\|}\Big)\|\vect{y}_{t,\perp'}\|\leq\Big(1+\frac{\sqrt{\rho\epsilon}}{2\ell}+\cdot\frac{8\rho r}{\sqrt{\rho\epsilon}}\Big)\|\vect{y}_{t+1,\perp'}\|,
\end{align}
which leads to
\begin{align}\label{eqn:GD-perp'-recursion}
\|\vect{y}_{t,\perp'}\|
\leq \|\vect{y}_{0,\perp'}\|(1+\sqrt{\rho\epsilon}/(2\ell)+{8\rho r}/\sqrt{\rho\epsilon})^t
\leq\|\vect{y}_{0,\perp'}\|(1+5\sqrt{\rho\epsilon}/(8\ell))^t,\ \ \forall t\in[\mathscr{T}].
\end{align}
Similarly, we can have the recurrence formula for $\|\vect{y}_{t,\parallel}\|$:
\begin{align}
\|\vect{y}_{t+1,\parallel}\|\geq (1+\sqrt{\rho\epsilon}/(2\ell))\|\vect{y}_{t,\parallel}\|-\|\Delta_{\parallel}\|\geq
(1+\sqrt{\rho\epsilon}/(2\ell)-\|\Delta\|/(\alpha_t\|\vect{y}_{t}\|))\|\vect{y}_{t,\parallel}\|.
\end{align}
Considering that $\|\Delta\|/\|\vect{y}_t\|\leq\rho r/\ell$ due to \eqn{gradient-approximation-0}, we can further have
\begin{align}
\|\vect{y}_{t+1,\parallel}\|
\geq
(1+\sqrt{\rho\epsilon}/(2\ell)-\rho r/(\alpha_t\ell))\|\vect{y}_{t,\parallel}\|.
\end{align}
Intuitively, $\|\vect{y}_{t,\parallel}\|$ should have a faster increasing rate than $\|\vect{y}_{t,\perp}\|$ in this gradient-based Hessian power method, ignoring the deviation from quadratic approximation. As a result, the value the value $\alpha_t=\|\vect{y}_{t,\parallel}\|/\|\vect{y}_t\|$ should be non-decreasing. It is demonstrated in \lem{GD-component-lowerbound} in \append{NC-GD} that, even if we count this deviation in, $\alpha_t$ can still be lower bounded by some constant $\alpha_{\min}$:
\begin{align}
\alpha_{t}\geq\alpha_{\min}=\frac{\delta_0}{4}\sqrt{\frac{\pi}{n}},\qquad \forall 1\leq t\leq\mathscr{T}.
\end{align}
by which we can further deduce that
\begin{align}
\|\vect{y}_{t,\parallel}\|\geq\|\vect{y}_{0,\parallel}\|(1+\sqrt{\rho\epsilon}/\ell-\rho r/(\alpha_{\min}\ell))^t
\geq\|\vect{y}_{0,\parallel}\|(1+7\sqrt{\rho\epsilon}/(8\ell))^t,\quad\forall 1\leq t\leq\mathscr{T}.
\end{align}
Observe that
\begin{align}
\frac{\|\vect{y}_{\mathscr{T},\perp'}\|}{\|\vect{y}_{\mathscr{T},\parallel}\|}
\leq\frac{\|\vect{y}_{0,\perp'}\|}{\|\vect{y}_{0,\parallel}\|}\cdot\Big(\frac{1+5\sqrt{\rho\epsilon}/(8\ell)}{1+7\sqrt{\rho\epsilon}/(8\ell)}\Big)^{\mathscr{T}}
\leq\frac{1}{\delta_{0}}\sqrt{\frac{n}{\pi}}\Big(\frac{1+5\sqrt{\rho\epsilon}/(8\ell)}{1+7\sqrt{\rho\epsilon}/(8\ell)}\Big)^{\mathscr{T}}\leq \frac{\sqrt{\rho\epsilon}}{8\ell}.
\end{align}
Since $\|\vect{y}_{\mathscr{T},\parallel}\|\leq\|\vect{y}_{\mathscr{T}}\|$, we have $\|\vect{y}_{\mathscr{T},\perp'}\|/\|\vect{y}_{\mathscr{T}}\|\leq\sqrt{\rho\epsilon}/(8\ell)$, contradiction. Hence, there here exists some $t_0$ with $1\leq t_0\leq \mathscr{T}$ such that $\|\vect{y}_{t_0,\perp'}\|/\|\vect{y}_{t_0}\|\leq\sqrt{\rho\epsilon}/(8\ell)$. Consider the normalized vector $\hat{\vect{e}}=\vect{y}_{t_0}/r$, we use $\hat{\vect{e}}_{\perp'}$ and $\hat{\vect{e}}_{\parallel'}$ to separately denote the component of $\hat{\vect{e}}$ in $\mathfrak{S}_{\perp}'$ and $\mathfrak{S}_{\parallel}'$. Then, $\|\hat{\vect{e}}_{\perp'}\|\leq\sqrt{\rho\epsilon}/(8\ell)$ whereas $\|\hat{\vect{e}}_{\parallel'}\|\geq 1-\rho\epsilon/(8\ell)^2$. Then,
\begin{align}
\hat{\vect{e}}^{T}\mathcal{H}(\vect{0})\hat{\vect{e}}=(\hat{\vect{e}}_{\perp'}+\hat{\vect{e}}_{\parallel'})^{T}\mathcal{H}(\vect{0})(\hat{\vect{e}}_{\perp'}+\hat{\vect{e}}_{\parallel'})=\hat{\vect{e}}_{\perp'}^{T}\mathcal{H}(\vect{0})\hat{\vect{e}}_{\perp'}+\hat{\vect{e}}_{\parallel'}^{T}\mathcal{H}(\vect{0})\hat{\vect{e}}_{\parallel'}
\end{align}
since $\mathcal{H}(\vect{0})\hat{\vect{e}}_{\perp'}\in\mathfrak{S}_{\perp}'$ and $\mathcal{H}(\vect{0})\hat{\vect{e}}_{\parallel'}\in\mathfrak{S}_{\parallel}'$. Due to the $\ell$-smoothness of the function, all eigenvalue of the Hessian matrix has its absolute value upper bounded by $\ell$. Hence,
\begin{align}
\hat{\vect{e}}_{\perp'}^{T}\mathcal{H}(\vect{0})\hat{\vect{e}}_{\perp'}\leq\ell\|\hat{\vect{e}}_{\perp'}^{T}\|_{2}^2=\rho\epsilon/(64\ell^2).
\end{align}
Further according to the definition of $\mathfrak{S}_{\parallel}$, we have
\begin{align}
\hat{\vect{e}}_{\parallel'}^{T}\mathcal{H}(\vect{0})\hat{\vect{e}}_{\parallel'}\leq-\sqrt{\rho\epsilon}\|\hat{\vect{e}}_{\parallel'}\|^2/2.
\end{align}
Combining these two inequalities together, we can obtain
\begin{align}
\hat{\vect{e}}^{T}\mathcal{H}(\vect{0})\hat{\vect{e}}&=\hat{\vect{e}}_{\perp}^{T}\mathcal{H}(\vect{0})\hat{\vect{e}}_{\perp'}+\hat{\vect{e}}_{\parallel'}^{T}\mathcal{H}(\vect{0})\hat{\vect{e}}_{\parallel'}
\leq-\sqrt{\rho\epsilon}\|\hat{\vect{e}}_{\parallel'}\|^2/2+\rho\epsilon/(64\ell^2)\leq-\sqrt{\rho\epsilon}/4.
\end{align}
\end{proof}

\begin{remark}
In practice, the value of $\|\vect{y}_t\|$ can become large during the execution of \algo{NC-finding}. To fix this, we can renormalize $\vect{y}_t$ to have $\ell_{2}$-norm $r$ at the ends of such iterations, and this does not influence the performance of the algorithm.
\end{remark}

%===============================================================================

\subsection{Faster negative curvature finding based on accelerated gradient descents}\label{sec:AGD+NC}

In this subsection, we replace the GD part in \algo{NC-finding} by AGD to obtain an accelerated negative curvature finding subroutine with similar effect and faster convergence rate, based on which we further implement our Accelerated Gradient Descent with Negative Curvature Finding (\algo{ANCGD}). Near any saddle point $\tilde{\vect{x}}\in\mathbb{R}^{n}$ with $\lambda_{\min}(\mathcal{H}(\tilde{\vect{x}}))\leq-\sqrt{\rho\epsilon}$, \algo{ANCGD} finds a unit vector $\hat{\vect{e}}$ such that $\hat{\vect{e}}^{T}\mathcal{H}(\tilde{\vect{x}})\hat{\vect{e}}\leq-\sqrt{\rho\epsilon}/4$.

\begin{algorithm}[htbp]
\caption{Perturbed Accelerated Gradient Descent with Accelerated Negative Curvature Finding($\vect{x}_0,\eta,\theta,\gamma,s,\mathscr{T}',r'$)}\label{algo:ANCGD}
$t_{\text{perturb}}\leftarrow 0$, $\vect{z}_0\leftarrow \vect{x}_0$, $\tilde{\vect{x}}\leftarrow \vect{x}_0$, $\zeta\leftarrow 0$\;
\For{$t=0,1,2,...,T$}{
\If{$\|\nabla f(\vect{x}_{t})\|\leq\epsilon$ and $t-t_{\text{perturb}}>\mathscr{T}$}{
$\tilde{\vect{x}}=\vect{x}_t$\;
$\vect{x}_t\leftarrow\text{Uniform}(\mathbb{B}_{\tilde{\vect{x}}}(r'))$ where $\text{Uniform}(\mathbb{B}_{\tilde{\vect{x}}}(r'))$ is the $\ell_2$-norm ball centered at $\tilde{\vect{x}}$ with radius $r'$, $\vect{z}_t\leftarrow \vect{x}_t$, $\zeta\leftarrow \nabla f(\tilde{\vect{x}})$, $t_{\text{perturb}}\leftarrow t$\;\label{lin:AGD-perturb}
}
\If{$t-t_{\text{perturb}}=\mathscr{T}'$}{
$\hat{\vect{e}}:=\frac{\vect{x}_t-\tilde{\vect{x}}}{\|\vect{x}_t-\tilde{\vect{x}}\|}$\;\label{lin:NC-direction}
$\vect{x}_{t}\leftarrow \tilde{\vect{x}}-\frac{f'_{\hat{\vect{e}}}(\tilde{\vect{x}})}{4|f'_{\hat{\vect{e}}}(\tilde{\vect{x}})|}\sqrt{\frac{\epsilon}{\rho}}\cdot\hat{\vect{e}}$, $\vect{z}_t\leftarrow\vect{x}_t$, $\zeta=\vect{0}$\;
}

$\vect{x}_{t+1}\leftarrow\vect{z}_t-\eta(\nabla f(\vect{z}_t)-\zeta)$\;
$\vect{v}_{t+1}\leftarrow\vect{x}_{t+1}-\vect{x}_{t}$\;
$\vect{z}_{t+1}\leftarrow\vect{x}_{t+1}+(1-\theta)\vect{v}_{t+1}$\;
\If{$t_{\text{perturb}}\neq 0$ and $t-t_{\text{perturb}}<\mathscr{T}'$}
{
$\vect{z}_{t+1}\leftarrow \tilde{\vect{x}}+r'\cdot\frac{\vect{z}_{t+1}-\tilde{\vect{x}}}{\|\vect{z}_{t+1}-\tilde{\vect{x}}\|}$, $\ \vect{x}_{t+1}\leftarrow\tilde{\vect{x}}+r'\cdot\frac{\vect{x}_{t+1}-\tilde{\vect{x}}}{\|\vect{z}_{t+1}-\tilde{\vect{x}}\|}$\;
}
\Else{
\If{$f(\vect{x}_{t+1})\leq f(\vect{z}_{t+1})+\left\<\nabla f(\vect{z}_{t+1}),\vect{x}_{t+1}-\vect{z}_{t+1}\right\>-\frac{\gamma}{2}\|\vect{z}_{t+1}-\vect{x}_{t+1}\|^2$}
{
$(\vect{x}_{t+1},\vect{v}_{t+1})\leftarrow$NegativeCurvatureExploitation($\vect{x}_{t+1},\vect{v}_{t+1},s$)\footnotemark\; \label{lin:NCE} 
$\vect{z}_{t+1}\leftarrow\vect{x}_{t+1}+(1-\theta)\vect{v}_{t+1}$\;
}
}
}
\end{algorithm}

\footnotetext{\label{fnote:NCE} This NegativeCurvatureExploitation (NCE) subroutine was originally introduced in \cite[Algorithm 3]{jin2018accelerated} and is called when we detect that the current momentum $\vect{v}_t$ coincides with a negative curvature direction of $\vect{z}_t$. In this case, we reset the momentum $\vect{v}_t$ and decide whether to exploit this direction based on the value of $\|\vect{v}_t\|$.}

The following proposition exhibits the effectiveness of \algo{ANCGD} for finding negative curvatures near saddle points:
\begin{restatable}{proposition}{ANC-finding}\label{prop:ANC-finding}
Suppose the function $f\colon\mathbb{R}^n\to\mathbb{R}$ is $\ell$-smooth and $\rho$-Hessian Lipschitz. For any $0<\delta_0\leq 1$, we specify our choice of parameters and constants we use as follows:
\begin{align}
    \eta    &:= \frac{1}{4\ell}     & \theta  &:= \frac{(\rho\epsilon)^{1/4}}{4\sqrt{\ell}}  &
      \mathscr{T}'  &:= \frac{32\sqrt{\ell}}{(\rho\epsilon)^{1/4}}\log\Big(\frac{\ell}{\delta_0}\sqrt{\frac{n}{\rho\epsilon}}\Big)  \nonumber\\
      \gamma   &:= \frac{\theta^2}{\eta}     & s &:= \frac{\gamma}{4\rho} &
      r'&:=\frac{\delta_0\epsilon}{32}\sqrt{\frac{\pi}{\rho n}}.
\label{eqn:parameter-choice-AGD}
\end{align}
Then for a point $\tilde{\vect{x}}$ satisfying $\lambda_{\min}(\nabla^2f(\tilde{\vect{x}}))\leq-\sqrt{\rho\epsilon}$, if running \algo{ANCGD} with the uniform perturbation in \lin{AGD-perturb} being added at $t=0$, the unit vector $\hat{\vect{e}}$ in  \lin{NC-direction} obtained after $\mathscr{T}'$ iterations satisfies:
\begin{align}
\mathbb{P}\Big(\hat{\vect{e}}^{T}\mathcal{H}(\vect{x})\hat{\vect{e}}\leq-\sqrt{\rho\epsilon}/4\Big)\geq 1-\delta_0.
\end{align}
\end{restatable}
The proof of \prop{ANC-finding} is similar to the proof of \prop{NC-finding}, and is deferred to \append{ANC-finding-proof}.

%==============================================================

\subsection{Escaping saddle points using negative curvature finding}\label{sec:NC-esc-saddle}
In this subsection, we demonstrate that our \algo{NC-finding} and \algo{ANCGD} with the ability to find negative curvature near saddle points can further escape saddle points of nonconvex functions. The intuition works as follows: we start with gradient descents or accelerated gradient descents until the gradient becomes small. At this position, we compute the negative curvature direction, described by a unit vector $\hat{\vect{e}}$, via \algo{NC-finding} or the negative curvature finding subroutine of \algo{ANCGD}. Then, we add a perturbation along this direction of negative curvature and go back to gradient descents or accelerated gradient descents with an additional NegativeCurvatureExploitation subroutine (see \fnote{NCE}). It has the following guarantee:
\begin{lemma}\label{lem:utilize-NC}
Suppose the function $f\colon\mathbb{R}^n\to\mathbb{R}$ is $\ell$-smooth and $\rho$-Hessian Lipschitz. Then for any point $\vect{x}_0\in\mathbb{R}^{n}$, if there exists a unit vector $\hat{\vect{e}}$ satisfying $\hat{\vect{e}}^{T}\mathcal{H}(\vect{x}_0)\hat{\vect{e}}\leq-\frac{\sqrt{\rho\epsilon}}{4}$ where $\mathcal{H}$ stands for the Hessian matrix of function $f$, the following inequality holds:
\begin{align}\label{eqn:function-decrease}
f\Big(\vect{x}_0-\frac{f'_{\hat{\vect{e}}}(\vect{x}_0)}{4|f'_{\hat{\vect{e}}}(\vect{x}_0)|}\sqrt{\frac{\epsilon}{\rho}}\cdot\hat{\vect{e}}\Big)\leq f(\vect{x}_0)-\frac{1}{384}\sqrt{\frac{\epsilon^3}{\rho}},
\end{align}
where $f'_{\hat{\vect{e}}}$ stands for the gradient component of $f$ along the direction of $\hat{\vect{e}}$.
\end{lemma}
\begin{proof}
Without loss of generality, we assume $\vect{x}_0=\vect{0}$. We can also assume $\left\<\nabla f(\vect{0}),\hat{\vect{e}}\right\>\leq 0$; if this is not the case we can pick $-\hat{\vect{e}}$ instead, which still satisfies $(-\hat{\vect{e}})^{T}\mathcal{H}(\vect{x}_0)(-\hat{\vect{e}})\leq-\frac{\sqrt{\rho\epsilon}}{4}$. In practice, to figure out whether we should use $\hat{\vect{e}}$ or $-\hat{\vect{e}}$, we apply both of them in \eqn{function-decrease} and choose the one with smaller function value. Then, for any $\vect{x}=x_{\hat{\vect{e}}}\hat{\vect{e}}$ with some $x_{\hat{\vect{e}}}>0$, we have $\frac{\partial^2 f}{\partial x_{\hat{\vect{e}}}^2}(\vect{x})\leq -\frac{\sqrt{\rho\epsilon}}{4}+\rho x_{\hat{\vect{e}}}$ due to the $\rho$-Hessian Lipschitz condition of $f$. Hence,
\begin{align}
\frac{\partial f}{\partial x_{\hat{\vect{e}}}}(\vect{x})\leq f'_{\hat{\vect{e}}}(\vect{0})-\frac{\sqrt{\rho\epsilon}}{4}x_{\hat{\vect{e}}}+\rho x_{\hat{\vect{e}}}^2,
\end{align}
by which we can further derive that
\begin{align}
f(x_{\hat{\vect{e}}}\hat{\vect{e}})-f(\vect{0})\leq f'_{\hat{\vect{e}}}(\vect{0})x_{\hat{\vect{e}}}-\frac{\sqrt{\rho\epsilon}}{8}x_{\hat{\vect{e}}}^2+\frac{\rho}{3}x_{\hat{\vect{e}}}^3\leq-\frac{\sqrt{\rho\epsilon}}{8}x_{\hat{\vect{e}}}^2+\frac{\rho}{3}x_{\hat{\vect{e}}}^3.
\end{align}
Settings $x_{\hat{\vect{e}}}=\frac{1}{4}\sqrt{\frac{\epsilon}{\rho}}$ gives \eqn{function-decrease}.
\end{proof}

We give the full algorithm details based on \algo{NC-finding} in \append{explicit-algorithms}. Along this approach, we achieve the following:
\begin{theorem}[informal, full version deferred to \append{PAGD+ANC}]\label{thm:PAGD+ANC-Complexity}
For any $\epsilon>0$ and a constant $0<\delta\leq 1$, \algo{ANCGD} satisfies that at least one of the iterations $\vect{x}_{t}$ will be an $\epsilon$-approximate second-order stationary point in
\begin{align}
\tilde{O}\Big(\frac{(f(\vect{x}_{0})-f^{*})}{\epsilon^{1.75}}\cdot\log n\Big)
\end{align}
iterations, with probability at least $1-\delta$, where $f^{*}$ is the global minimum of $f$.
\end{theorem}
Intuitively, the proof of \thm{PAGD+ANC-Complexity} has two parts. The first part is similar to the proof of {\cite[Theorem 3]{jin2018accelerated}}, which shows that PAGD uses $\tilde{O}(\log^6 n/\epsilon^{1.75})$ iterations to escape saddle points. We show that there can be at most $\tilde{O}(\Delta_f/\epsilon^{1.75})$ iterations with the norm of gradient larger than $\epsilon$ using almost the same techniques, but with slightly different parameter choices. The second part is based on the negative curvature part of \algo{ANCGD}, our accelerated negative curvature finding algorithm. Specifically, at each saddle point we encounter, we can take $\tilde{O}(\log n/\epsilon^{1/4})$ iterations to find its negative curvature (\prop{ANC-finding}), and add a perturbation in this direction to decrease the function value by $O(\epsilon^{1.5})$ (\lem{utilize-NC}). Hence, the iterations introduced by \algo{ANC-finding} can be at most $\tilde{O}\big(\frac{\log n}{\epsilon^{1.5}}\cdot\frac{1}{\epsilon^{0.25}}\big)=\tilde{O}(\log n/\epsilon^{1.75})$, which is simply an upper bound on the overall iteration number. The detailed proof is deferred to \append{PAGD+ANC}.

\begin{remark}
Although \thm{PAGD+ANC-Complexity} only demonstrates that our algorithms will visit some $\epsilon$-approximate second-order stationary point during their execution with high probability, it is straightforward to identify one of them if we add a termination condition: once Negative Curvature Finding (\algo{NC-finding} or \algo{ANCGD}) is applied, we record the position $\vect{x}_{t_0}$ and the function value decrease due to the following perturbation. If the function value decrease is larger than $\frac{1}{384}\sqrt{\frac{\epsilon^3}{\rho}}$ as per \lem{utilize-NC}, then the algorithms make progress. Otherwise, $\vect{x}_{t_0}$ is an $\epsilon$-approximate second-order stationary point with high probability.
\end{remark}

%%%%%%%%%%%%%%%%%%%%%%%%%%%%%%%%%%%%%%%%%%%%%%%%%%%%%%%%%%%%%%%%%%%%%%%%%%%%%%%%%

\section{Stochastic setting}\label{sec:SNC-setting}

In this section, we present a stochastic version of \algo{NC-finding} using stochastic gradients, and demonstrate that it can also be used to escape saddle points and obtain a polynomial speedup in $\log n$ compared to the perturbed stochastic gradient (PSGD) algorithm in \cite{jin2019stochastic}.
\subsection{Stochastic negative curvature finding}\label{sec:SGD+NC}
In the stochastic gradient descent setting, the exact gradients oracle $\nabla f$ of function $f$ cannot be accessed. Instead, we only have unbiased stochastic gradients $\vect{g}(\vect{x};\theta)$ such that
\begin{align}
\nabla f(\vect{x})=\mathbb{E}_{\theta\sim\mathcal{D}}[\vect{g}(\vect{x};\theta)]\qquad\forall\vect{x}\in\mathbb{R}^n,
\end{align}
where $\mathcal{D}$ stands for the probability distribution followed by the random variable $\theta$. Define
\begin{align}
\zeta(\vect{x};\theta):=g(\vect{x};\theta)-\nabla f(\vect{x})
\end{align}
to be the error term of the stochastic gradient. Then, the expected value of vector $\zeta(\vect{x};\theta)$ at any $\vect{x}\in\mathbb{R}^n$ equals to $\vect{0}$. Further, we assume the stochastic gradient $\vect{g}(\vect{x},\theta)$ also satisfies the following assumptions, which were also adopted in previous literatures; see e.g.~\cite{fang2018spider,jin2017escape,jin2019stochastic,zhou2018finding}.
\begin{assumption}\label{assum:stochastic-variance}
For any $\vect{x}\in\mathbb{R}^{n}$, the stochastic gradient $\vect{g}(\vect{x};\theta)$ with $\theta\sim\mathcal{D}$ satisfies:
\begin{align}
\Pr[(\|\vect{g}(\vect{x};\theta)-\nabla f(\vect{x})\|\geq t)]\leq 2\exp(-t^2/(2\sigma^2)),\quad\forall t\in\mathbb{R}.
\end{align}
\end{assumption}
\begin{assumption}\label{assum:stochastic-lipschitz}
For any $\theta\in\text{supp}(\mathcal{D})$, $\vect{g}(\vect{x};\theta)$ is $\tilde{\ell}$-Lipschitz for some constant $\tilde{\ell}$:
\begin{align}
\|\vect{g}(\vect{x}_1;\theta)-\vect{g}(\vect{x}_2;\theta)\|\leq\tilde{\ell}\|\vect{x}_1-\vect{x}_2\|,\qquad\forall\vect{x}_1,\vect{x}_2\in\mathbb{R}^n.
\end{align}
\end{assumption}
\assum{stochastic-lipschitz} emerges from the fact that the stochastic gradient $\vect{g}$ is often obtained as an exact gradient of some smooth function,
\begin{align}
\vect{g}(\vect{x};\theta)=\nabla f(\vect{x};\theta).
\end{align}
In this case, \assum{stochastic-lipschitz} guarantees that for any $\theta\sim\mathcal{D}$, the spectral norm of the Hessian of $f(\vect{x};\theta)$ is upper bounded by $\tilde{\ell}$.
Under these two assumptions, we can construct the stochastic version of \algo{NC-finding}, as shown in \algo{SNC-finding}.
\begin{algorithm}[htbp]
\caption{Stochastic Negative Curvature Finding($\vect{x}_0,r_s,\mathscr{T}_s,m$).}
\label{algo:SNC-finding}
$\vect{y}_0\leftarrow0,\ L_0\leftarrow r_s$\;
\For{$t=1,...,\mathscr{T}_s$}{
Sample $\left\{\theta^{(1)},\theta^{(2)},\cdots,\theta^{(m)}\right\}\sim\mathcal{D}$\;
$\vect{g}(\vect{y}_{t-1})\leftarrow\frac{1}{m}\sum_{j=1}^m\big(\vect{g}(\vect{x}_0+\vect{y}_{t-1};\theta^{(j)})-\vect{g}(\vect{x}_0;\theta^{(j)})\big)$\;
$\vect{y}_t\leftarrow\vect{y}_{t-1}-\frac{1}{\ell}(\vect{g}(\vect{y}_{t-1})+\xi_t/L_{t-1}),\qquad\xi_t\sim\mathcal{N}\Big(0,\frac{r_s^2}{d}I\Big)$\;
$L_t\leftarrow\frac{\|\vect{y}_t\|}{r_s}L_{t-1}$ and $\vect{y}_t\leftarrow \vect{y}_t\cdot\frac{r_s}{\|\vect{y}_t\|}$\; \label{lin:y_t-renormalization-stochastic}
}
\textbf{Output} $\vect{y}_{\mathscr{T}}/r_s$.
\end{algorithm}
\vspace{-1mm}

Similar to the non-stochastic setting, \algo{SNC-finding} can be used to escape from saddle points and obtain a polynomial speedup in $\log n$ compared to PSGD algorithm in \cite{jin2019stochastic}. This is quantitatively shown in the following theorem:
\begin{theorem}[informal, full version deferred to \append{PSGD+SNC}]\label{thm:PSGD+SNC-Complexity}
For any $\epsilon>0$ and a constant $0<\delta\leq 1$, our algorithm\footnote{Our algorithm based on \algo{SNC-finding} has similarities to the Neon2$^{\text{online}}$ algorithm in \cite{allen2018neon2}. Both algorithms find a second-order stationary point for stochastic optimization in $\tilde{O}(\log^2 n/\epsilon^4)$ iterations, and we both apply directed perturbations based on the results of negative curvature finding. Nevertheless, our algorithm enjoys simplicity by only having a single loop, whereas Neon2$^{\text{online}}$ has a nested loop for boosting their confidence.} based on \algo{SNC-finding} using only stochastic gradient descent satisfies that at least one of the iterations $\vect{x}_{t}$ will be an $\epsilon$-approximate second-order stationary point in
\begin{align}
\tilde{O}\Big(\frac{(f(\vect{x}_{0})-f^{*})}{\epsilon^{4}}\cdot\log^2 n\Big)
\end{align}
iterations, with probability at least $1-\delta$, where $f^{*}$ is the global minimum of $f$.
\end{theorem}

%%%%%%%%%%%%%%%%%%%%%%%%%%%%%%%%%%%%%%%%%%%%%%%%%%%%%%%

\section{Numerical experiments}\label{sec:numerical}

In this section, we provide numerical results that exhibit the power of our negative curvature finding algorithm for escaping saddle points. More experimental results can be found in \append{more-numerics}. All the experiments are performed on MATLAB R2019b on a computer with Six-Core Intel Core i7 processor and 16GB memory, and their codes are given in the supplementary material.
\vspace{-1mm}

\paragraph{Comparison between \algo{NC-finding} and PGD.} We compare the performance of our \algo{NC-finding} with the perturbed gradient descent (PGD) algorithm in \cite{jin2019stochastic} on a test function $f(x_1,x_2)=\frac{1}{16}x_1^4-\frac{1}{2}x_1^2+\frac{9}{8}x_2^2$ with a saddle point at $(0,0)$. The advantage of \algo{NC-finding} is illustrated in \fig{quartic_descent}.

\begin{figure}[htbp]
\centering\includegraphics[width = 12.7cm]{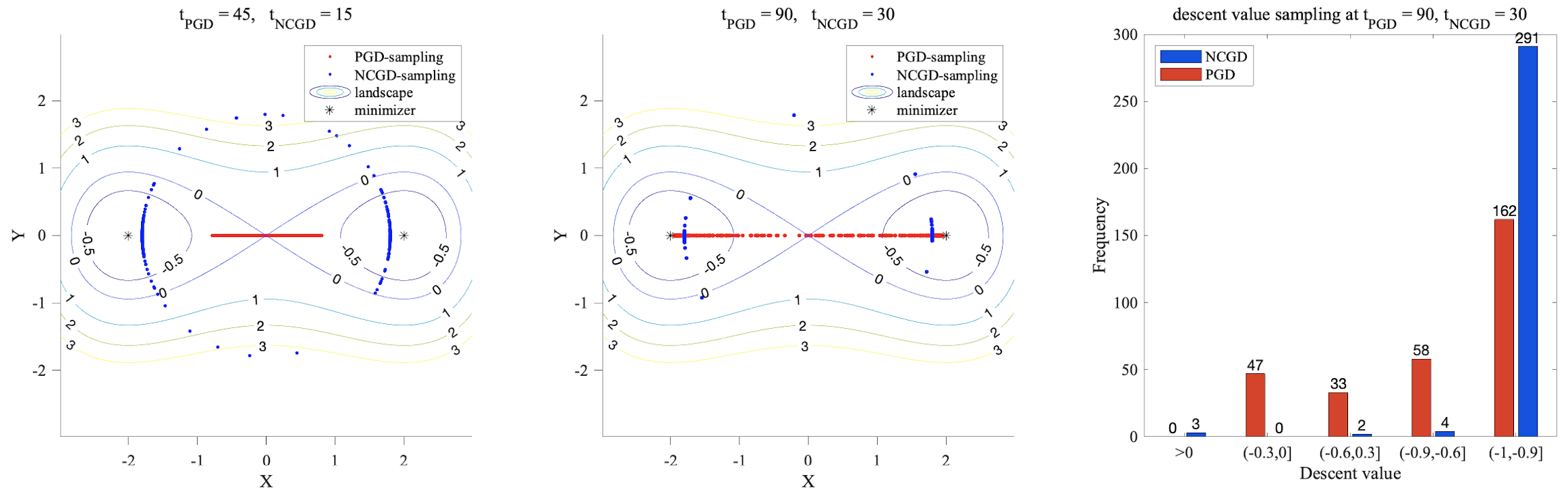}
\caption{Run \algo{NC-finding} and PGD on landscape $f(x_1,x_2)=\frac{1}{16}x_1^4-\frac{1}{2}x_1^2+\frac{9}{8}x_2^2$. Parameters: $\eta =0.05$ (step length), $r = 0.1$ (ball radius in PGD and parameter $r$ in \algo{NC-finding}), $M = 300$ (number of samplings).\\
\textbf{Left}: The contour of the landscape is placed on the background with labels being function values. Blue points represent samplings of \algo{NC-finding} at time step $t_{\text{NCGD}}=15$ and $t_{\text{NCGD}}=30$, and red points represent samplings of PGD at time step $t_{\text{PGD}}=45$ and $t_{\text{PGD}}=90$. \algo{NC-finding} transforms an initial uniform-circle distribution into a distribution concentrating on two points indicating negative curvature, and these two figures represent intermediate states of this process. It converges faster than PGD even when $t_{\text{NCGD}}\ll t_{\text{PGD}}$.\\
\textbf{Right}: A histogram of descent values obtained by \algo{NC-finding} and PGD, respectively. Set $t_{\text{NCGD}} = 30$ and $t_{\text{PGD}} = 90$. Although we run three times of iterations in PGD, there are still over $40\%$ of gradient descent paths with function value decrease no greater than $0.9$, while this ratio for \algo{NC-finding} is less than $5\%$.}
\label{fig:quartic_descent}
\end{figure}
\vspace{-2mm}

\paragraph{Comparison between \algo{SNC-finding} and PSGD.} We compare the performance of our \algo{SNC-finding} with the perturbed stochastic gradient descent (PSGD) algorithm in \cite{jin2019stochastic} on a test function $f(x_1,x_2)=(x_1^3-x_2^3)/2-3x_1x_2+(x_1^2+x_2^2)^2/2$. Compared to the landscape of the previous experiment, this function is more deflated from a quadratic field due to the cubic terms. Nevertheless, \algo{SNC-finding} still possesses a notable advantage compared to PSGD as demonstrated in \fig{stochastic_cubic_descent}.
\begin{figure}[htbp]
\centering\includegraphics[width = 12.7cm]{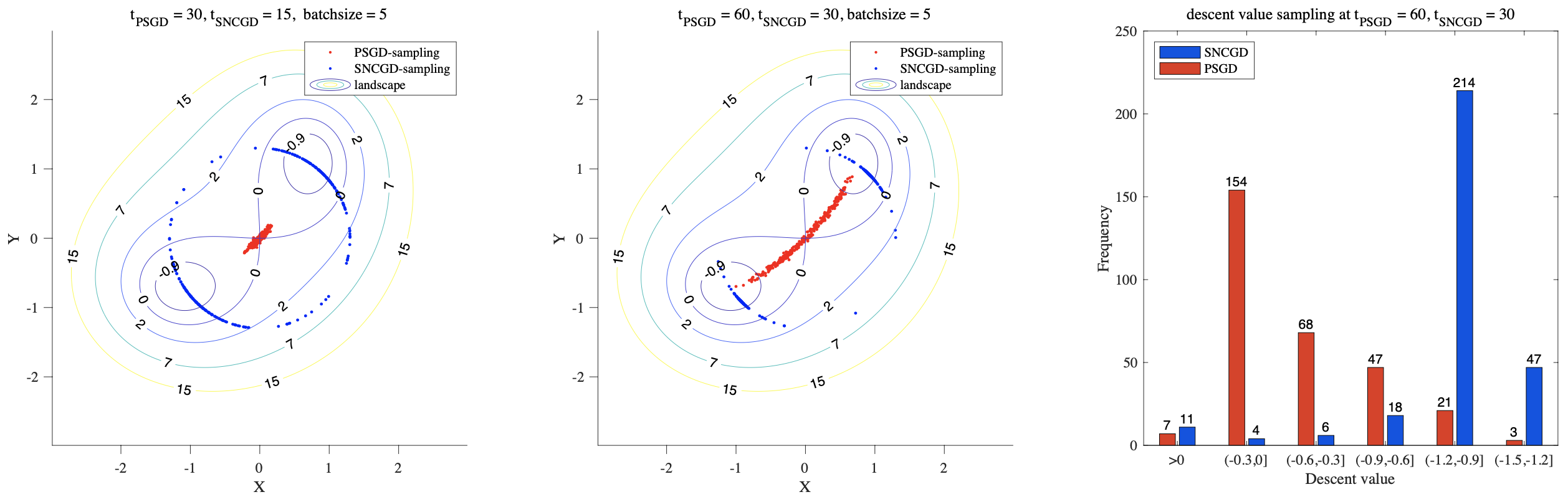}
\caption{Run \algo{SNC-finding} and PSGD on landscape $f(x_1,x_2)=\frac{x_1^3-x_2^3}{2}-3x_1x_2+\frac{1}{2}(x_1^2+x_2^2)^2$. Parameters: $\eta =0.02$ (step length), $r = 0.01$ (variance in PSGD and $r_s$ in \algo{SNC-finding}), $M = 300$ (number of samplings).\\
\textbf{Left}: The contour of the landscape is placed on the background with labels being function values. Blue points represent samplings of \algo{SNC-finding} at time step $t_{\text{SNCGD}}=15$ and $t_{\text{SNCGD}}=30$, and red points represent samplings of PSGD at time step $t_{\text{PSGD}}=30$ and $t_{\text{PSGD}}=60$. \algo{SNC-finding} transforms an initial uniform-circle distribution into a distribution concentrating on two points indicating negative curvature, and these two figures represent intermediate states of this process. It converges faster than PSGD even when $t_{\text{SNCGD}}\ll t_{\text{PSGD}}$.\\
\textbf{Right}: A histogram of descent values obtained by \algo{SNC-finding} and PSGD, respectively. Set $t_{\text{SNCGD}} = 30$ and $t_{\text{PSGD}} = 60$. Although we run two times of iterations in PSGD, there are still over $50\%$ of SGD paths with function value decrease no greater than $0.6$, while this ratio for \algo{SNC-finding} is less than $10\%$.}
\label{fig:stochastic_cubic_descent}
\end{figure}

%%%%%%%%%%%%%%%%%%%%%%%%%%%%%%%%%%%%%%%%%%%%%%%%%%%%%%%%%%%%%%%%%%%%%%%%%%%%%%

\section*{Acknowledgements}
We thank Jiaqi Leng for valuable suggestions and generous help on the design of numerical experiments. TL was supported by the NSF grant PHY-1818914 and a Samsung Advanced Institute of Technology Global Research Partnership.

%%%%%%%%%%%%%%%%%%%%%%%%%%%%%%%%%%%%%%%%%%%%%%%%%%%%%%%%%%%%%%%%%%%

%%%%%%%%%%%%%%%%%%%%%%%%%%%%%%%%%%%%%%%%%%%%%%%%%%%%%%%%%%%%%%%%%%%%%%%%%%%%%%

\appendix
\newpage
\normalsize

\section{Auxiliary lemmas}\label{append:existing-lemma}
In this section, we introduce auxiliary lemmas that are necessary for our proofs.
\begin{lemma}[{\cite[Lemma 19]{jin2019stochastic}}]\label{lem:descent-lemma}
If $f(\cdot)$ is $\ell$-smooth and $\rho$-Hessian Lipschitz, $\eta=1/\ell$, then the gradient descent sequence $\{\vect{x}_{t}\}$ obtained by $\vect{x}_{t+1}:=\vect{x}_t-\eta\nabla f(\vect{x}_t)$ satisfies:
\begin{align}
f(\vect{x}_{t+1})-f(\vect{x}_{t})\leq \eta \|\nabla f(\vect{x})\|^{2}/2,
\end{align}
for any step $t$ in which Negative Curvature Finding is not called.
\end{lemma}

\begin{lemma}[{\cite[Lemma 23]{jin2019stochastic}}]\label{lem:stochastic-descent-lemma}
For a $\ell$-smooth and $\rho$-Hessian Lipschitz function $f$ with its stochastic gradient satisfying \assum{stochastic-variance}, there exists an absolute constant $c$ such that, for any fixed $t,t_0,\iota>0$, with probability at least $1-4e^{\iota}$, the stochastic gradient descent sequence in \algo{SGD+NC} satisfies:
\begin{align}
f(\vect{x}_{t_0+t})-f(\vect{x}_{t_0})\leq-\frac{\eta}{8}\sum_{i=0}^{t-1}\|\nabla f(\vect{x}_{t_0+i})\|^2+c\cdot\frac{\sigma^2}{\ell}(t+\iota),
\end{align}
if during $t_0\sim t_0+t$, Stochastic Negative Curvature Finding has not been called.
\end{lemma}

\begin{lemma}[{\cite[Lemma 29]{jin2019stochastic}}]\label{lem:stochastic-ab}
Denote $\alpha(t):=\Big[\sum_{\tau=0}^{t-1}(1+\eta\gamma)^{2(t-1-\tau)}\Big]^{1/2}$ and $\beta(t):=(1+\eta\gamma)^t/\sqrt{2\eta\gamma}$. If $\eta\gamma\in[0,1]$, then we have
\begin{enumerate}
\item $\alpha(t)\leq\beta(t)$ for any $t\in\mathbb{N}$;
\item$\alpha(t)\geq\beta(t)/\sqrt{3},\quad\forall t\geq\ln 2/(\eta\gamma)$.
\end{enumerate}
\end{lemma}

\begin{lemma}[{\cite[Lemma 30]{jin2019stochastic}}]\label{lem:stochastic-qp-norm}
Under the notation of \lem{stochastic-ab} and \append{SNC-finding}, letting $-\gamma:=\lambda_{\min}(\tilde{\mathcal{H}})$, for any $0\leq t\leq\mathscr{T}_s$ and $\iota>0$ we have
\begin{align}
\Pr\Big(\|\vect{q}_p(t)\|\leq\beta(t)\eta r_s\cdot\sqrt{\iota}\Big)\geq1-2e^{-\iota},
\end{align}
and
\begin{align}
&\Pr\Big(\|\vect{q}_p(t)\|\geq\frac{\beta(t)\eta r_s}{10\sqrt{n}}\cdot\frac{\delta}{\mathscr{T}_s}\Big)\geq 1-\frac{\delta}{\mathscr{T}_s},\\
&\Pr\Big(\|\vect{q}_p(t)\|\geq\frac{\beta(t)\eta r_s}{10\sqrt{n}}\cdot\delta\Big)\geq 1-\delta.
\end{align}
\end{lemma}

\begin{definition}[{\cite[Definition 32]{jin2019stochastic}}]\label{defn:nSG}
A random vector $\vect{X}\in\mathbb{R}^n$ is norm-subGaussian (or nSG($\sigma$)), if there exists $\sigma$ so that:
\begin{align}
\Pr(\|\vect{X}-\mathbb{E}[\vect{X}]\|\geq t)\leq 2e^{-\frac{t^2}{2\sigma^2}},\quad\forall t\in\mathbb{R}.
\end{align}
\end{definition}

\begin{lemma}[{\cite[Lemma 37]{jin2019stochastic}}]\label{lem:stochastic-contraction-37}
Given i.i.d. $\vect{X}_1,\ldots,\vect{X}_{\tau}\in\mathbb{R}^n$ all being zero-mean nSG$(\sigma_i)$ defined in \defn{nSG}, then for any $\iota>0$, and $B>b>0$, there exists an absolute constant $c$ such that, with probability at least $1-2n\log(B/b)\cdot e^{-\iota}$:
\begin{align}
\sum_{i=1}^n\sigma_i^2\geq B\quad\text{or}\quad\big\|\sum_{i=1}^{\tau}\vect{X}_i\big\|\leq c\cdot\sqrt{\max\left\{\sum_{i=1}^{\tau}\sigma_i^2,b\right\}\cdot\iota}.
\end{align}
\end{lemma}

The next two lemmas are frequently used in the large gradient scenario of the accelerated gradient descent method:
\begin{lemma}[{\cite[Lemma 7]{jin2018accelerated}}]\label{lem:AGD-large-gradient}
Consider the setting of \thm{PAGD+ANC-full}, define a new parameter
\begin{align}
\tilde{\mathscr{T}}:=\frac{\sqrt{\ell}}{(\rho\epsilon)^{1/4}}\cdot c_A,
\end{align}
for some large enough constant $c_A$. If $\|\nabla f(\vect{x}_{\tau})\|\geq \epsilon$ for all $\tau\in[0,\tilde{\mathscr{T}}]$, then there exists a large enough positive constant $c_{A0}$, such that if we choose $c_{A}\geq c_{A0}$, by running \algo{ANCGD} we have $E_{\tilde{\mathscr{T}}}-E_{0}\leq -\mathscr{E}$, in which $\mathscr{E}=\sqrt{\frac{\epsilon^{3}}{\rho}}\cdot c_{A}^{-7}$, and $E_{\tau}$ is defined as:
\begin{align}\label{eqn:Lemma-A7}
E_{\tau}:=f(\vect{x}_{\tau})+\frac{1}{2\eta}\|\vect{v}_{\tau}\|^{2}.
\end{align}
\end{lemma}

\begin{lemma}[{\cite[Lemma 4 and Lemma 5]{jin2018accelerated}}]\label{lem:Hamiltonian-decrease}
Assume that the function $f$ is $\ell$-smooth. Consider the setting of \thm{PAGD+ANC-full}, for every iteration that is not within $\mathscr{T}'$ steps after uniform perturbation, we have:
\begin{align}
E_{\tau+1}\leq E_{\tau},
\end{align}
where $E_{\tau}$ is defined in Eq.~\eqn{Lemma-A7} in \lem{AGD-large-gradient}.
\end{lemma}

%%%%%%%%%%%%%%%%%%%%%%%%%%%%%%%%%%%%%%%%%%%%%%%%

\section{Proof details of negative curvature finding by gradient descents}\label{append:NC-GD}
\subsection{Proof of \lem{GD-component-lowerbound}}\label{append:NC-GD-proof}
\begin{restatable}{lemma}{GDcomponentlowerbound}
\label{lem:GD-component-lowerbound}
Under the setting of \prop{NC-finding}, we use $\alpha_{t}$ to denote
\begin{align}
\alpha_{t}=\|\vect{y}_{t,\parallel}\|/\|\vect{y}_{t}\|,
\end{align}
where $\vect{y}_{t,\parallel}$ defined in Eq.~\eqn{para-perp-defn} is the component of $\vect{y}_t$ in the subspace $\mathfrak{S}_{\parallel}$ spanned by $\left\{\vect{u}_1,\vect{u}_2,\ldots,\vect{u}_p\right\}$. Then, during all the $\mathscr{T}$ iterations of \algo{NC-finding}, we have $\alpha_{t}\geq\alpha_{\min}$ for
\begin{align}
\alpha_{\min}=\frac{\delta_0}{4}\sqrt{\frac{\pi}{n}},
\end{align}
given that $\alpha_{0}\geq\sqrt{\frac{\pi}{n}}\delta_0$.
\end{restatable}

To prove \lem{GD-component-lowerbound}, we first introduce the following lemma:
\begin{lemma}\label{lem:GD-worstcase}
Under the setting of \prop{NC-finding} and \lem{GD-component-lowerbound}, for any $t>0$ and initial value $\vect{y}_0$, $\alpha_{t}$ achieves its minimum possible value in the specific landscapes satisfying
\begin{align}\label{eqn:worst-case}
\lambda_1=\lambda_2=\cdots=\lambda_p=\lambda_{p+1}=\lambda_{p+2}=\cdots=\lambda_{n-1}=-\sqrt{\rho\epsilon}.
\end{align}
\end{lemma}
\begin{proof}
We prove this lemma by contradiction. Suppose for some $\tau>0$ and initial value $\vect{y}_0$, $\alpha_\tau$ achieves its minimum value in some landscape $g$ that does not satisfy Eq.~\eqn{worst-case}. That is to say, there exists some $k\in[n-1]$ such that $\lambda_k\neq-\sqrt{\rho\epsilon}$.

We first consider the case $k>p$. Since $\lambda_{n-1},\ldots,\lambda_{p+1}\geq-\sqrt{\rho\epsilon}$, we have $\lambda_{k}>-\sqrt{\rho\epsilon}$. We use $\left\{\vect{y}_{g,t}\right\}$ to denote the iteration sequence of $\vect{y}_t$ in this landscape.

Consider another landscape $h$ with $\lambda_k=-\sqrt{\rho\epsilon}$ and all other eigenvalues being the same as $g$, and we use $\left\{\vect{y}_{h,t}\right\}$ to denote the iteration sequence of $\vect{y}_t$ in this landscape. Furthermore, we assume that at each gradient query, the deviation from pure quadratic approximation defined in Eq.~\eqn{Delta-defn}, denoted $\Delta_{h,t}$, is the same as the corresponding deviation $\Delta_{g,t}$ in the landscape $g$ along all the directions other than $\vect{u}_k$. As for its component $\Delta_{h,t,k}$ along $\vect{u}_k$, we set $|\Delta_{h,t,k}|=|\Delta_{g,t,k}|$ with its sign being the same as $y_{t,k}$. 

Under these settings, we have $y_{h,\tau,j}=y_{g,\tau,j}$ for any $j\in[n]$ and $j\neq k$. As for the component along $\vect{u}_k$, we have $|y_{h,\tau,j}|>|y_{g,\tau,j}|$. Hence, by the definition of $\vect{y}_{t,\parallel}$ in Eq.~\eqn{para-perp-defn}, we have
\begin{align}
\|\vect{y}_{g,\tau,\parallel}\|=\big(\sum_{j=1}^p y_{g,\tau,j}^2\big)^{1/2}=\big(\sum_{j=1}^p y_{h,\tau,j}^2\big)^{1/2}=\|\vect{y}_{h,\tau,\parallel}\|,
\end{align}
whereas
\begin{align}
\|\vect{y}_{g,\tau}\|=\big(\sum_{j=1}^n y_{g,\tau,j}^2\big)^{1/2}<\big(\sum_{j=1}^n y_{h,\tau,j}^2\big)^{1/2}=\|\vect{y}_{h,\tau}\|,
\end{align}
indicating
\begin{align}
\alpha_{g,\tau}=\frac{\|\vect{y}_{g,\tau,\parallel}\|}{\|\vect{y}_{g,\tau}\|}
>\frac{\|\vect{y}_{h,\tau,\parallel}\|}{\|\vect{y}_{h,\tau}\|}=\alpha_{h,\tau},
\end{align}
contradicting to the supposition that $\alpha_\tau$ achieves its minimum possible value in $g$. Similarly, we can also obtain a contradiction for the case $k\leq p$.
\end{proof}

Equipped with \lem{GD-worstcase}, we are now ready to prove \lem{GD-component-lowerbound}.

\begin{proof}
In this proof, we consider the worst case, where the initial value $\alpha_{0}=\sqrt{\frac{\pi}{n}}\delta_0$. Also, according to \lem{GD-worstcase}, we assume that the eigenvalues satisfy
\begin{align}
\lambda_1=\lambda_2=\cdots=\lambda_p=\lambda_{p+1}=\lambda_{p+2}=\cdots=\lambda_{n-1}=-\sqrt{\rho\epsilon}.
\end{align}
Moreover, we assume that each time we make a gradient call at some point $\vect{x}$, the derivation term $\Delta$ from pure quadratic approximation
\begin{align}
\Delta=\nabla h_f(\vect{x})-\mathcal{H}(\vect{0})\vect{x}
\end{align}
lies in the direction that can make $\alpha_{t}$ as small as possible. Then, the component $\Delta_{\perp}$ in $\mathfrak{S}_{\perp}$ should be in the direction of $\vect{x}_{\perp}$, and the component $\Delta_{\parallel}$ in $\mathfrak{S}_{\parallel}$ should be in the opposite direction to $\vect{x}_{\parallel}$, as long as $\|\Delta_{\parallel}\|\leq\|\vect{x}_{\parallel}\|$. Hence in this case, we have $\|\vect{y}_{t,\perp}\|/\|\vect{y}_t\|$ being non-decreasing. Also, it admits the following recurrence formula:
\begin{align}
\|\vect{y}_{t+1,\perp}\|&=(1+\sqrt{\rho\epsilon}/\ell)\|\vect{y}_{t,\perp}\|+\|\Delta_{\perp}\|/\ell\\
&\leq (1+\sqrt{\rho\epsilon}/\ell)\|\vect{y}_{t,\perp}\|+\|\Delta\|/\ell\\
&\leq\Big(1+\sqrt{\rho\epsilon}/\ell+\frac{\|\Delta\|}{\ell\|\vect{y}_{t,\perp}\|}\Big)\|\vect{y}_{t,\perp}\|,
\end{align}
where the second inequality is due to the fact that $\nu$ can be an arbitrarily small positive number. Note that since $\|\vect{y}_{t,\perp}\|/\|\vect{y}_t\|$ is non-decreasing in this worst-case scenario, we have
\begin{align}
\frac{\|\Delta\|}{\|\vect{y}_{t,\perp}\|}\leq \frac{\|\Delta\|}{\|\vect{y}_{t}\|}\cdot\frac{\|\vect{y}_0\|}{\|\vect{y}_{0,\perp}\|}\leq \frac{2\|\Delta\|}{\|\vect{y}_{t}\|}\leq 2\rho r,
\end{align}
which leads to
\begin{align}\label{eqn:GD-perp-recurrence}
\|\vect{y}_{t+1,\perp}\|\leq(1+\sqrt{\rho\epsilon}/\ell+2\rho r/\ell)\|\vect{y}_{t,\perp}\|.
\end{align}
On the other hand, suppose for some value $t$, we have $\alpha_{k}\geq\alpha_{\min}$ with any $1\leq k\leq t$. Then,
\begin{align}
\|\vect{y}_{t+2,\parallel}\|&=(1+\sqrt{\rho\epsilon}/\ell)\|\vect{y}_{t+1,\perp}\|-\|\Delta_{\parallel}\|/\ell\\
&\geq\Big(1+\sqrt{\rho\epsilon}/\ell-\frac{\|\Delta\|}{\ell\|\vect{y}_{t,\parallel}\|}\Big)\|\vect{y}_{t,\parallel}\|.
\end{align}
Note that since $\|\vect{y}_{t,\parallel}\|/\|\vect{y}_t\|\geq\alpha_{\min}$, we have
\begin{align}
\frac{\|\Delta\|}{\|\vect{y}_{t,\parallel}\|}\geq \frac{\|\Delta\|}{\alpha_{\min}\|\vect{y}_{t}\|}=\rho r/\alpha_{\min},
\end{align}
which leads to
\begin{align}\label{eqn:GD-parallel-recurrence}
\|\vect{y}_{t+1,\parallel}\|\geq(1+\sqrt{\rho\epsilon}/\ell-\rho r/(\alpha_{\min}\ell))\|\vect{y}_{t,\parallel}\|.
\end{align}
Then we can observe that
\begin{align}
\frac{\|\vect{y}_{t,\parallel}\|}{\|\vect{y}_{t,\perp}\|}\geq
\frac{\|\vect{y}_{0,\parallel}\|}{\|\vect{y}_{0,\perp}\|}\cdot\Big(\frac{1+\sqrt{\rho\epsilon}/\ell-\rho r/(\alpha_{\min}\ell)}{1+\sqrt{\rho\epsilon}/\ell+2\rho r/\ell}\Big)^t,
\end{align}
where
\begin{align}
\frac{1+\sqrt{\rho\epsilon}/\ell-\rho r/(\alpha_{\min}\ell)}{1+\sqrt{\rho\epsilon}/\ell+2\rho r/\ell}
&\geq(1+\sqrt{\rho\epsilon}/\ell-\rho r/(\alpha_{\min}\ell))(1-\sqrt{\rho\epsilon}/\ell-2\rho r/\ell)\\
&\geq 1-\rho\epsilon/\ell^2-\frac{2\rho r}{\alpha_{\min}\ell}\geq 1-\frac{1}{\mathscr{T}},
\end{align}
by which we can derive that
\begin{align}
\frac{\|\vect{y}_{t,\parallel\|}}{\|\vect{y}_{t,\perp}\|}
&\geq \frac{\|\vect{y}_{0,\parallel}\|}{\|\vect{y}_{0,\perp}\|}(1-1/\mathscr{T})^{t}\\
&\geq \frac{\|\vect{y}_{0,\parallel}\|}{\|\vect{y}_{0,\perp}\|}\exp\Big(-\frac{t}{\mathscr{T}-1}\Big)\geq\frac{\|\vect{y}_{0,\parallel}\|}{2\|\vect{y}_{0,\perp}\|},
\end{align}
indicating
\begin{align}
\alpha_{t}=\frac{\|\vect{y}_{t,\parallel}\|}{\sqrt{\|\vect{y}_{t,\parallel}\|^2+\|\vect{y}_{t,\perp}\|^2}}\geq \frac{\|\vect{y}_{0,\parallel}\|}{4\|\vect{y}_{0,\perp}\|}\geq\alpha_{\min}.
\end{align}
Hence, as long as $\alpha_{k}\geq\alpha_{\min}$ for any $1\leq k\leq t-1$, we can also have $\alpha_{t}\geq \alpha_{\min}$ if $t\leq \mathscr{T}$. Since we have $\alpha_{0}\geq \alpha_{\min}$, we can thus claim that $\alpha_{t}\geq \alpha_{\min}$ for any $t\leq \mathscr{T}$ using recurrence.
\end{proof}

%===================================================

\subsection{Proof of \prop{ANC-finding}}\label{append:ANC-finding-proof}
To make it easier to analyze the properties and running time of \algo{ANCGD}, we introduce a new \algo{ANC-finding},
\begin{algorithm}[htbp]
\caption{Accelerated Negative Curvature Finding without Renormalization($\tilde{\vect{x}},r',\mathscr{T}'$).}
\label{algo:ANC-finding}
$\vect{x}_0\leftarrow$Uniform$(\mathbb{B}_0(r'))$ where $\mathbb{B}_0(r')$ is the $\ell_{2}$-norm ball centered at $\tilde{\vect{x}}$ with radius $r'$\;
$\vect{z}_0\leftarrow\vect{x}_0$\;
\For{$t=1,...,\mathscr{T}'$}{
$\vect{x}_{t+1}\leftarrow\vect{z}_t-\eta\Big(\frac{\|\vect{z}_t-\tilde{\vect{x}}\|}{r'}\nabla f\Big(r'\cdot\frac{\vect{z}_t-\tilde{\vect{x}}}{\|\vect{z}_t-\tilde{\vect{x}}\|}+\tilde{\vect{x}}\Big)-\nabla f(\tilde{\vect{x}})\Big)$\;
$\vect{v}_{t+1}\leftarrow\vect{x}_{t+1}-\vect{x}_{t}$\;
$\vect{z}_{t+1}\leftarrow\vect{x}_{t+1}+(1-\theta)\vect{v}_{t+1}$\;

}
\textbf{Output} $\frac{\vect{x}_{\mathscr{T}'}-\tilde{\vect{x}}}{\|\vect{x}_{\mathscr{T}'}-\tilde{\vect{x}}\|}$.
\end{algorithm}
which has a more straightforward structure and has the same effect as \algo{ANCGD} near any saddle point $\tilde{\vect{x}}$. Quantitatively, this is demonstrated in the following lemma:
\begin{lemma}\label{lem:ANC-same}
Under the setting of \prop{ANC-finding}, suppose the perturbation in \lin{AGD-perturb} of \algo{ANCGD} is added at $t=0$. Then with the same value of $r'$, $\mathscr{T}'$, $\tilde{\vect{x}}$ and $\vect{x}_0$, the output of \algo{ANC-finding} is the same as the unit vector $\hat{\vect{e}}$ in \algo{ANCGD} obtained $\mathscr{T}'$ steps later. In other words, if we separately denote the iteration sequence of $\left\{\vect{x}_t\right\}$ in \algo{ANCGD} and \algo{ANC-finding} as
\begin{align}
\left\{\vect{x}_{1,0},\vect{x}_{1,1},\ldots,\vect{x}_{1,\mathscr{T}'}\right\},\qquad
\left\{\vect{x}_{2,0},\vect{x}_{2,1},\ldots,\vect{x}_{2,\mathscr{T}'}\right\},
\end{align}
we have
\begin{align}
\frac{\vect{x}_{1,\mathscr{T}'}-\tilde{\vect{x}}}{\|\vect{x}_{1,\mathscr{T}'}-\tilde{\vect{x}}\|}=\frac{\vect{x}_{2,\mathscr{T}'}-\tilde{\vect{x}}}{\|\vect{x}_{2,\mathscr{T}'}-\tilde{\vect{x}}\|}.
\end{align}
\end{lemma}
\begin{proof}
Without loss of generality, we assume $\tilde{\vect{x}}=\vect{0}$ and $\nabla f(\tilde{\vect{x}})=\vect{0}$. Use recurrence to prove the desired relationship. Suppose the following identities holds for all $k\leq t$ with $t$ being some natural number:
\begin{align}\label{eqn:AGD-recurrence}
\frac{\vect{x}_{2,k}}{\|\vect{x}_{2,k}\|}=\frac{\vect{x}_{1,k}}{r},
\qquad\frac{\vect{z}_{2,k}}{\|\vect{x}_{2,k}\|}=\frac{\vect{z}_{1,k}}{r'}.
\end{align}
Then,
\begin{align}
\vect{x}_{2,t+1}=\vect{z}_{2,t}-\eta\cdot\frac{\|\vect{z}_{2,t}\|}{r'}\nabla f(\vect{z}_{2,t}\cdot r'/\|\vect{z}_{2,t}\|)=\frac{\|\vect{z}_{2,t}\|}{r'}(\vect{z}_{1,t}-\eta\nabla f(\vect{z}_{1,t})).
\end{align}
Adopting the notation in \algo{ANCGD}, we use $\vect{x}_{1,t+1}'$ and $\vect{z}_{1,t+1}'$ to separately denote the value of $\vect{x}_{1,t+1}$ and $\vect{z}_{1,t+1}$ before renormalization:
\begin{align}
\vect{x}_{1,t+1}'=\vect{z}_{1,t}-\eta\nabla f(\vect{z}_{1,t}),\quad
\vect{z}_{1,t+1}'=\vect{x}_{1,t+1}'+(1-\theta)(\vect{x}_{1,t+1}'-\vect{x}_{1,t}).
\end{align}
Then,
\begin{align}
\vect{x}_{2,t+1}=\frac{\|\vect{z}_{2,t}\|}{r'}(\vect{z}_{1,t}-\eta \nabla f(\vect{z}_{1,t}))=\frac{\|\vect{z}_{2,t}\|}{r'}\cdot\vect{x}_{1,t+1}',
\end{align}
which further leads to
\begin{align}\label{eqn:y-recurrence}
\vect{z}_{2,t+1}=\vect{x}_{2,t+1}+(1-\theta)(\vect{x}_{2,t+1}-\vect{x}_{2,t})=\frac{\|\vect{z}_{2,t}\|}{r'}\cdot\vect{z}_{1,t+1}'.
\end{align}
Note that $\vect{z}_{1,t+1}=\frac{r'}{\|\vect{z}_{1,t+1}'\|}\cdot\vect{z}_{1,t+1}'$, we thus have
\begin{align}
\frac{\vect{z}_{2,t+1}}{\|\vect{z}_{2,t+1}\|}=\frac{\vect{z}_{1,t+1}}{r'}.
\end{align}
Hence,
\begin{align}
\vect{x}_{2,t+1}=\frac{\|\vect{z}_{2,t}\|}{r'}\cdot\vect{x}_{1,t+1}'=\frac{\|\vect{z}_{2,t}\|}{r'}\cdot\frac{\|\vect{z}_{1,t+1}\|}{\|\vect{z}_{1,t}\|}\cdot\vect{x}_{1,t+1}=\frac{\|\vect{z}_{2,t+1}\|}{r'}\cdot \vect{x}_{1,t+1}.
\end{align}
Since \eqn{AGD-recurrence} holds for $k=0$, we can now claim that it also holds for $k=\mathscr{T}'$.
\end{proof}
\lem{ANC-same} shows that, \algo{ANC-finding} also works in principle for finding the negative curvature near any saddle point $\tilde{\vect{x}}$. But considering that \algo{ANC-finding} may result in an exponentially large $\|\vect{x}_{t}\|$ during execution, and it is hard to be merged with the AGD algorithm for large gradient scenarios. Hence, only \algo{ANCGD} is applicable in practical situations.

Use $\mathcal{H}(\tilde{\vect{x}})$ to denote the Hessian matrix of $f$ at $\tilde{\vect{x}}$. Observe that $\mathcal{H}(\tilde{\vect{x}})$ admits the following eigen-decomposition:
\begin{align}\label{eqn:AGD-Hessian-decomposition}
\mathcal{H}(\tilde{\vect{x}})=\sum_{i=1}^{n}\lambda_i\vect{u}_{i}\vect{u}_i^{T},
\end{align}
where the set $\{\vect{u}_i\}_{i=1}^{n}$ forms an orthonormal basis of $\mathbb{R}^n$. Without loss of generality, we assume the eigenvalues $\lambda_1,\lambda_2,\ldots,\lambda_n$ corresponding to $\vect{u}_1,\vect{u}_2,\ldots,\vect{u}_n$ satisfy
\begin{align}
\lambda_1\leq\lambda_2\leq\cdots\leq\lambda_n,
\end{align}
in which $\lambda_1\leq-\sqrt{\rho\epsilon}$. If $\lambda_n\leq-\sqrt{\rho\epsilon}/2$, \prop{ANC-finding} holds directly, since no matter the value of $\hat{\vect{e}}$, we can have $f(\vect{x}_{\mathscr{T}'})-f(\tilde{\vect{x}})\leq-\sqrt{\epsilon^3/\rho}/384$. Hence, we only need to prove the case where $\lambda_n>-\sqrt{\rho\epsilon}/2$, where there exists some $p>1$ with
\begin{align}
\lambda_p\leq -\sqrt{\rho\epsilon}\leq \lambda_{p+1}.
\end{align}
We use $\mathfrak{S}_{\parallel}$ to denote the subspace of $\mathbb{R}^{n}$ spanned by $\left\{\vect{u}_1,\vect{u}_2,\ldots,\vect{u}_p\right\}$, and use $\mathfrak{S}_{\perp}$ to denote the subspace spanned by $\left\{\vect{u}_{p+1},\vect{u}_{p+2},\ldots,\vect{u}_{n}\right\}$. Then we can have the following lemma:
\begin{lemma}\label{lem:AGD-component-lowerbound}
Under the setting of \prop{ANC-finding}, we use $\alpha_{t}'$ to denote
\begin{align}
\alpha_{t}'=\frac{\|\vect{x}_{t,\parallel}-\tilde{\vect{x}}_{\parallel}\|}{\|\vect{x}_{t}-\tilde{\vect{x}}\|},
\end{align}
in which $\vect{x}_{t,\parallel}$ is the component of $\vect{x}_t$ in the subspace $\mathfrak{S}_{\parallel}$. Then, during all the $\mathscr{T}'$ iterations of \algo{ANC-finding}, we have $\alpha_{t}'\geq\alpha_{\min}'$ for
\begin{align}
\alpha_{\min}'=\frac{\delta_0}{8}\sqrt{\frac{\pi}{n}},
\end{align}
given that $\alpha_{0}'\geq\sqrt{\frac{\pi}{n}}\delta_0$.
\end{lemma}

\begin{proof}
Without loss of generality, we assume $\tilde{\vect{x}}=\vect{0}$ and $\nabla f(\tilde{\vect{x}})=\vect{0}$. In this proof, we consider the worst case, where the initial value $\alpha_{0}'=\sqrt{\frac{\pi}{n}}\delta_0$ and the component $x_{0,n}$ along $\vect{u}_n$ equals 0. In addition, according to the same proof of \lem{GD-worstcase}, we assume that the eigenvalues satisfy
\begin{align}
\lambda_1=\lambda_2=\lambda_3=\cdots=\lambda_p=\lambda_{p+1}=\lambda_{p+2}=\cdots=\lambda_{n-1}=-\sqrt{\rho\epsilon}.
\end{align}
Out of the same reason, we assume that each time we make a gradient call at point $\vect{z}_t$, the derivation term $\Delta$ from pure quadratic approximation
\begin{align}
\Delta=\frac{\|\vect{z}_t\|}{r'}\cdot\Big(\nabla f(\vect{z}_t\cdot r'/\|\vect{z}_t\|)-\mathcal{H}(\vect{0})\cdot\frac{r'}{\|\vect{z}_t\|}\cdot\vect{z}_t\Big)
\end{align}
lies in the direction that can make $\alpha_{t}'$ as small as possible. Then, the component $\Delta_{\parallel}$ in $\mathfrak{S}_{\parallel}$ should be in the opposite direction to $\vect{v}_{\parallel}$, and the component $\Delta_{\perp}$ in $\mathfrak{S}_{\perp}$ should be in the direction of $\vect{v}_{\perp}$. Hence in this case, we have both $\|\vect{x}_{t,\perp}\|/\|\vect{x}_t\|$ and $\|\vect{z}_{t,\perp}\|/\|\vect{z}_t\|$ being non-decreasing. Also, it admits the following recurrence formula:
\begin{align}
\|\vect{x}_{t+2,\perp}\|&\leq (1+\eta\sqrt{\rho\epsilon})\big(\|\vect{x}_{t+1,\perp}\|+(1-\theta)(\|\vect{x}_{t+1,\perp}\|-\|\vect{x}_{t,\perp}\|)\big)+\eta\|\Delta_{\perp}\|.
\end{align}
Since $\|\vect{x}_{t,\perp}\|/\|\vect{x}_t\|$ is non-decreasing in this worst-case scenario, we have
\begin{align}
\frac{\|\Delta_{\perp}\|}{\|\vect{x}_{t+1,\perp}\|}\leq \frac{\|\Delta\|}{\|\vect{x}_{t+1}\|}\cdot\frac{\|\vect{x}_0\|}{\|\vect{x}_{0,\perp}\|}\leq \frac{2\|\Delta\|}{\|\vect{x}_{t+1}\|}\leq 2\rho r',
\end{align}
which leads to
\begin{align}\label{eqn:perp-recurrence}
\|\vect{x}_{t+2,\perp}\|\leq(1+\eta\sqrt{\rho\epsilon}+2\eta\rho r')\big((2-\theta)\|\vect{x}_{t+1,\perp}\|-(1-\theta)\|\vect{x}_{t,\perp}\|\big).
\end{align}
On the other hand, suppose for some value $t$, we have $\alpha_{k}'\geq\alpha_{\min}'$ with any $1\leq k\leq t+1$. Then,
\begin{align}
\|\vect{x}_{t+2,\parallel}\|&\geq(1+\eta(\sqrt{\rho\epsilon}-\nu))\big(\|\vect{x}_{t+1,\parallel}\|+(1-\theta)(\|\vect{x}_{t+1,\parallel}\|-\|\vect{x}_{t,\parallel}\|)\big)+\eta\|\Delta_{\parallel}\|\\
&\geq(1+\eta\sqrt{\rho\epsilon})\big(\|\vect{x}_{t+1,\parallel}\|+(1-\theta)(\|\vect{x}_{t+1,\parallel}\|-\|\vect{x}_{t,\parallel}\|)\big)-\eta\|\Delta\|.
\end{align}
Note that since $\|\vect{x}_{t+1,\parallel}\|/\|\vect{x}_t\|\geq\alpha_{\min}'$ for all $t>0$, we also have
\begin{align}
\frac{\|\vect{x}_{t+1,\parallel}\|}{\|\vect{x}_t\|}\geq\alpha_{\min}',\quad\forall t\geq 0,
\end{align}
which further leads to 
\begin{align}
\frac{\|\Delta\|}{\|\vect{z}_{t+1,\parallel}\|}\geq \frac{\|\Delta\|}{\alpha_{\min}'\|\vect{z}_{t+1}\|}=\rho r'/\alpha_{\min}',
\end{align}
which leads to
\begin{align}\label{eqn:parallel-recurrence}
\|\vect{x}_{t+2,\parallel}\|\geq(1+\eta\sqrt{\rho\epsilon}-\eta\rho r'/\alpha_{\min}')\big((2-\theta)\|\vect{x}_{t+1,\parallel}\|-(1-\theta)\|\vect{x}_{t,\parallel}\|\big).
\end{align}
Consider the sequences with recurrence that can be written as
\begin{align}
\xi_{t+2}=(1+\kappa)\big((2-\theta)\xi_{t+1}-(1-\theta)\xi_{t}\big)
\end{align}
for some $\kappa>0$. Its characteristic equation can be written as
\begin{align}
x^2-(1+\kappa)(2-\theta)x+(1+\kappa)(1-\theta)=0,
\end{align}
whose roots satisfy
\begin{align}
x=\frac{1+\kappa}{2}\Big((2-\theta)\pm\sqrt{(2-\theta)^2-\frac{4(1-\theta)}{1+\kappa}}\Big),
\end{align}
indicating
\begin{align}
\xi_{t}=\Big(\frac{1+\kappa}{2}\Big)^t\big(C_1(2-\theta+\mu)^t+C_2(2-\theta-\mu)^t\big),
\end{align}
where $\mu:=\sqrt{(2-\theta)^2-\frac{4(1-\theta)}{1+\kappa}}$, for constants $C_1$ and $C_2$ being
\begin{align}
\left\{
\begin{aligned}
&C_1=-\frac{2-\theta-\mu}{2\mu}\xi_0+\frac{1}{(1+\kappa)\mu}\xi_1,
\\
&C_2=\frac{2-\theta+\mu}{2\mu}\xi_0-\frac{1}{(1+\kappa)\mu}\xi_1.
\end{aligned}\right.
\end{align}
Then by the inequalities \eqn{perp-recurrence} and \eqn{parallel-recurrence}, as long as $\alpha_{k}'\geq\alpha_{\min}'$ for any $1\leq k\leq t-1$, the values $\|\vect{x}_{t,\perp}\|$ and $\|\vect{x}_{t,\parallel}\|$ satisfy
\begin{align}
\|\vect{x}_{t,\perp}\|&\leq
\Big(-\frac{2-\theta-\mu_{\perp}}{2\mu_{\perp}}\xi_{0,\perp}+\frac{1}{(1+\kappa_{\perp})\mu_{\perp}}\xi_{1,\perp}\Big)\cdot
\Big(\frac{1+\kappa_{\perp}}{2}\Big)^t\cdot(2-\theta+\mu_{\perp})^t\\
&\quad+\Big(\frac{2-\theta+\mu_{\perp}}{2\mu_{\perp}}\xi_{0,\perp}-\frac{1}{(1+\kappa_{\perp})\mu_{\perp}}\xi_{1,\perp}\Big)\cdot
\Big(\frac{1+\kappa_{\perp}}{2}\Big)^t\cdot(2-\theta-\mu_{\perp})^t,
\end{align}
and
\begin{align}
\|\vect{x}_{t,\parallel}\|&\geq
\Big(-\frac{2-\theta-\mu_{\parallel}}{2\mu_{\parallel}}\xi_{0,\parallel}+\frac{1}{(1+\kappa_{\parallel})\mu_{\parallel}}\xi_{1,\parallel}\Big)\cdot
\Big(\frac{1+\kappa_{\parallel}}{2}\Big)^t\cdot(2-\theta+\mu_{\parallel})^t\\
&\quad+\Big(\frac{2-\theta+\mu_{\parallel}}{2\mu_{\parallel}}\xi_{0,\parallel}-\frac{1}{(1+\kappa_{\parallel})\mu_{\parallel}}\xi_{1,\parallel}\Big)\cdot
\Big(\frac{1+\kappa_{\parallel}}{2}\Big)^t\cdot(2-\theta-\mu_{\parallel})^t,
\end{align}
where
\begin{align}
\kappa_{\perp}&=\eta\sqrt{\rho\epsilon}+2\eta\rho r',
\qquad &\xi_{0,\perp}&=\|\vect{x}_{0,\perp}\|,
\qquad &\xi_{1,\perp}&=(1+\kappa_{\perp})\xi_{0,\perp},\\
\kappa_{\parallel}&=\eta\sqrt{\rho\epsilon}-\eta\rho r'/\alpha_{\min}',
\qquad &\xi_{0,\parallel}&=\|\vect{x}_{0,\parallel}\|,
\qquad &\xi_{1,\parallel}&=(1+\kappa_{\parallel})\xi_{0,\parallel}.
\end{align}
Furthermore, we can derive that
\begin{align}
\|\vect{x}_{t,\perp}\|&\leq
\Big(-\frac{2-\theta-\mu_{\perp}}{2\mu_{\perp}}\xi_{0,\perp}+\frac{1}{(1+\kappa_{\perp})\mu_{\perp}}\xi_{1,\perp}\Big)\cdot
\Big(\frac{1+\kappa_{\perp}}{2}\Big)^t\cdot(2-\theta+\mu_{\perp})^t\\
&\quad+\Big(\frac{2-\theta+\mu_{\perp}}{2\mu_{\perp}}\xi_{0,\perp}-\frac{1}{(1+\kappa_{\perp})\mu_{\perp}}\xi_{1,\perp}\Big)\cdot
\Big(\frac{1+\kappa_{\perp}}{2}\Big)^t\cdot(2-\theta+\mu_{\perp})^t\\
&\leq \xi_{0,\perp}\cdot\Big(\frac{1+\kappa_{\perp}}{2}\Big)^t\cdot(2-\theta+\mu_{\perp})^t=\|\vect{x}_{0,\perp}\|\cdot\Big(\frac{1+\kappa_{\perp}}{2}\Big)^t\cdot(2-\theta+\mu_{\perp})^t,
\end{align}
and
\begin{align}
\|\vect{x}_{t,\parallel}\|&\geq
\Big(-\frac{2-\theta-\mu_{\parallel}}{2\mu_{\parallel}}\xi_{0,\parallel}+\frac{1}{(1+\kappa_{\parallel})\mu_{\parallel}}\xi_{1,\parallel}\Big)\cdot
\Big(\frac{1+\kappa_{\parallel}}{2}\Big)^t\cdot(2-\theta+\mu_{\parallel})^t\\
&\quad+\Big(\frac{2-\theta+\mu_{\parallel}}{2\mu_{\parallel}}\xi_{0,\parallel}-\frac{1}{(1+\kappa_{\parallel})\mu_{\parallel}}\xi_{1,\parallel}\Big)\cdot
\Big(\frac{1+\kappa_{\parallel}}{2}\Big)^t\cdot(2-\theta-\mu_{\parallel})^t\\
&\geq\Big(-\frac{2-\theta-\mu_{\parallel}}{2\mu_{\parallel}}\xi_{0,\parallel}+\frac{1}{(1+\kappa_{\parallel})\mu_{\parallel}}\xi_{1,\parallel}\Big)\cdot
\Big(\frac{1+\kappa_{\parallel}}{2}\Big)^t\cdot(2-\theta+\mu_{\parallel})^t\\
&=\frac{\mu_{\parallel}+\theta}{2\mu_{\parallel}}\xi_{0,\parallel}\cdot\Big(\frac{1+\kappa_{\parallel}}{2}\Big)^t\cdot(2-\theta+\mu_{\parallel})^t\\
&\geq\frac{\|\vect{x}_{0,\parallel}\|}{2}\cdot\Big(\frac{1+\kappa_{\parallel}}{2}\Big)^t\cdot(2-\theta+\mu_{\parallel})^t.
\end{align}
Then we can observe that
\begin{align}
\frac{\|\vect{x}_{t,\parallel}\|}{\|\vect{x}_{t,\perp}\|}\geq \frac{\|\vect{x}_{0,\parallel}\|}{2\|\vect{x}_{0,\perp}\|}\cdot\Big(\frac{1+\kappa_{\parallel}}{1+\kappa_{\perp}}\Big)^t\cdot\Big(\frac{2-\theta+\mu_{\parallel}}{2-\theta+\mu_{\perp}}\Big)^t,
\end{align}
where
\begin{align}
\frac{1+\kappa_{\parallel}}{1+\kappa_{\perp}}&\geq (1+\kappa_{\parallel})(1-\kappa_{\perp})\\
&\geq 1-(2+1/\alpha_{\min}')\eta\rho r'-\kappa_{\parallel}\kappa_{\perp}\\
&\geq 1-2\eta\rho r'/\alpha'_{\min},
\end{align}
and
\begin{align}
\frac{2-\theta+\mu_{\parallel}}{2-\theta+\mu_{\perp}}
&=\frac{1+\mu_{\parallel}/(2-\theta)}{1+\mu_{\perp}/(2-\theta)}\\
&=\frac{1+\sqrt{1-\frac{4(1-\theta)}{(1+\kappa_{\parallel})(2-\theta)^2}}}{1+\sqrt{1-\frac{4(1-\theta)}{(1+\kappa_{\perp})(2-\theta)^2}}}\\
&\geq \Big(1+\frac{1}{2-\theta}\sqrt{\frac{\theta^2+\kappa_{\parallel}(2-\theta)^2}{1+\kappa_{\parallel}}}\Big)\Big(1-\frac{1}{2-\theta}\sqrt{\frac{\theta^2+\kappa_{\perp}(2-\theta)^2}{1+\kappa_{\perp}}}\Big)\\
&\geq 1-\frac{2(\kappa_{\perp}-\kappa_{\parallel})}{\theta}\geq 1-\frac{3\eta\rho r'}{\alpha_{\min}'\theta},
\end{align}
by which we can derive that
\begin{align}
\frac{\|\vect{x}_{t,\parallel\|}}{\|\vect{x}_{t,\perp}\|}
&\geq \frac{\|\vect{x}_{0,\parallel}\|}{2\|\vect{x}_{0,\perp}\|}\cdot\Big(1-\frac{4\rho r'}{\alpha_{\min}'\theta}\Big)^t\\
&\geq \frac{\|\vect{x}_{0,\parallel}\|}{2\|\vect{x}_{0,\perp}\|}(1-1/\mathscr{T}')^{t}\\
&\geq \frac{\|\vect{x}_{0,\parallel}\|}{2\|\vect{x}_{0,\perp}\|}\exp\Big(-\frac{t}{\mathscr{T}'-1}\Big)\geq\frac{\|\vect{x}_{0,\parallel}\|}{4\|\vect{x}_{0,\perp}\|},
\end{align}
indicating
\begin{align}
\alpha_{t}'=\frac{\|\vect{x}_{t,\parallel}\|}{\sqrt{\|\vect{x}_{t,\parallel}\|^2+\|\vect{x}_{t,\perp}\|^2}}\geq \frac{\|\vect{x}_{0,\parallel}\|}{8\|\vect{x}_{0,\perp}\|}\geq \alpha_{\min}'.
\end{align}
Hence, as long as $\alpha_{k}'\geq\alpha_{\min}'$ for any $1\leq k\leq t-1$, we can also have $\alpha_{t}'\geq \alpha_{\min}'$ if $t\leq \mathscr{T}'$. Since we have $\alpha_{0}'\geq \alpha_{\min}'$ and $\alpha_{1}'\geq \alpha_{\min}'$, we can claim that $\alpha_{t}'\geq \alpha_{\min}'$ for any $t\leq \mathscr{T}'$ using recurrence.
\end{proof}

Equipped with \lem{AGD-component-lowerbound}, we are now ready to prove \prop{ANC-finding}.
\begin{proof}
By \lem{ANC-same}, the unit vector $\hat{\vect{e}}$ in  \lin{NC-direction} of \algo{ANCGD} obtained after $\mathscr{T}'$ iterations equals to the output of \algo{ANC-finding} starting from $\tilde{\vect{x}}$. Hence in this proof we consider the output of \algo{ANC-finding} instead of the original \algo{ANCGD}.

If $\lambda_n\leq-\sqrt{\rho\epsilon}/2$, \prop{ANC-finding} holds directly. Hence, we only need to prove the case where $\lambda_n>-\sqrt{\rho\epsilon}/2$, in which there exists some $p'$ with
\begin{align}
\lambda_p'\leq -\sqrt{\rho\epsilon}/2< \lambda_{p+1}.
\end{align}
We use $\mathfrak{S}_{\parallel}'$, $\mathfrak{S}_{\perp}'$ to denote the subspace of $\mathbb{R}^{n}$ spanned by $\left\{\vect{u}_1,\vect{u}_2,\ldots,\vect{u}_{p'}\right\}$, $\left\{\vect{u}_{p'+1},\vect{u}_{p+2},\ldots,\vect{u}_{n}\right\}$. Furthermore, we define $\vect{x}_{t,\parallel'}:=\sum_{i=1}^{p'}\left\<\vect{u}_i,\vect{x}_t\right\>\vect{u}_i$,
$\vect{x}_{t,\perp'}:=\sum_{i=p'}^n\left\<\vect{u}_i,\vect{x}_t\right\>\vect{u}_i$, $\vect{v}_{t,\parallel'}:=\sum_{i=1}^{p'}\left\<\vect{u}_i,\vect{v}_t\right\>\vect{u}_i$,
$\vect{v}_{t,\perp'}:=\sum_{i=p'}^n\left\<\vect{u}_i,\vect{v}_t\right\>\vect{u}_i$
respectively to denote the component of $\vect{x}_t'$ and $\vect{v}_t'$ in \algo{ANC-finding} in the subspaces $\mathfrak{S}_{\parallel}'$, $\mathfrak{S}_{\perp}'$, and let $\alpha_{t}':=\|\vect{x}_{t,\parallel}\|/\|\vect{x}_{t}\|$.  Consider the case where $\alpha_{0}'\geq\sqrt{\frac{\pi}{n}}\delta_0$, which can be achieved with probability
\begin{align}
\Pr\left\{\alpha_0'\geq\sqrt{\frac{\pi}{n}}\delta_0\right\}\geq1-\sqrt{\frac{\pi}{n}}\delta_0\cdot\frac{\text{Vol}(\mathbb{B}_0^{n-1}(1))}{\text{Vol}(\mathbb{B}_0^{n}(1))}\geq 1-\sqrt{\frac{\pi}{n}}\delta_0\cdot\sqrt{\frac{n}{\pi}}=1-\delta_0,
\end{align}
 we prove that there exists some $t_0$ with $1\leq t_0\leq \mathscr{T}'$ such that
\begin{align}
\frac{\|\vect{x}_{t_0,\perp'}\|}{\|\vect{x}_{t_0}\|}\leq\frac{\sqrt{\rho\epsilon}}{8\ell}.
\end{align}
Assume the contrary, for any $t$ with $1\leq t\leq\mathscr{T}'$, we all have $\frac{\|\vect{x}_{t,\perp'}\|}{\|\vect{x}_{t}\|}>\frac{\sqrt{\rho\epsilon}}{8\ell}$ and $\frac{\|\vect{z}_{t,\perp'}\|}{\|\vect{z}_{t}\|}>\frac{\sqrt{\rho\epsilon}}{8\ell}$. Focus on the case where $\|\vect{x}_{t,\perp'}\|$, the component of $\vect{x}_{t}$ in subspace $\mathfrak{S}_{\perp}'$, achieves the largest value possible. Then in this case, we have the following recurrence formula:
\begin{align}
\|\vect{x}_{t+2,\perp'}\|\leq (1+\eta\sqrt{\rho\epsilon}/2)\big(\|\vect{x}_{t+1,\perp'}\|+(1-\theta)(\|\vect{x}_{t+1,\perp'}\|-\|\vect{x}_{t,\perp'}\|)\big)+\eta\|\Delta_{\perp'}\|.
\end{align}
Since $\frac{\|\vect{z}_{k,\perp'}\|}{\|\vect{z}_{k}\|}\geq \frac{\sqrt{\rho\epsilon}}{8\ell}$ for any $1\leq k\leq t+1$, we can derive that
\begin{align}
\frac{\|\Delta_{\perp}\|}{\|\vect{x}_{t+1,\perp}\|+(1-\theta)(\|\vect{x}_{t+1,\perp}\|-\|\vect{x}_{t,\perp}\|)}\leq\frac{\|\Delta\|}{\|\vect{z}_{t,\perp'}\|}\leq\frac{2\rho r'}{\sqrt{\rho\epsilon}},
\end{align}
which leads to
\begin{align}
\|\vect{x}_{t+2,\perp'}\|
&\leq (1+\eta\sqrt{\rho\epsilon}/2)\big(\|\vect{x}_{t+1,\perp'}\|+(1-\theta)(\|\vect{x}_{t+1,\perp'}\|-\|\vect{x}_{t,\perp'}\|)\big)+\eta\|\Delta_{\perp'}\|\\
&\leq (1+\eta\sqrt{\rho\epsilon}/2+2\rho r'/\sqrt{\rho\epsilon})\big((2-\theta)\|\vect{x}_{t+1,\perp'}\|-(1-\theta)\|\vect{x}_{t,\perp'}\|\big).
\end{align}
Using similar characteristic function techniques shown in the proof of \lem{AGD-component-lowerbound}, it can be further derived that
\begin{align}\label{eqn:perp'-recursion}
\|\vect{x}_{t,\perp'}\|\leq \|\vect{x}_{0,\perp'}\|\cdot\Big(\frac{1+\kappa_{\perp'}}{2}\Big)^t\cdot(2-\theta+\mu_{\perp'})^t,
\end{align}
for $\kappa_{\perp'}=\eta\sqrt{\rho\epsilon}/2+2\rho r'/\sqrt{\rho\epsilon}$ and $\mu_{\perp'}=\sqrt{(2-\theta)^2-\frac{4(1-\theta)}{1+\kappa_{\perp'}}}$, given $\frac{\|\vect{x}_{k,\perp'}\|}{\|\vect{x}_{k}\|}\geq \frac{\sqrt{\rho\epsilon}}{8\ell}$ and $\frac{\|\vect{z}_{k,\perp'}\|}{\|\vect{z}_{k}\|}\geq \frac{\sqrt{\rho\epsilon}}{8\ell}$ for any $1\leq k\leq t-1$. Due to \lem{AGD-component-lowerbound},
\begin{align}
\alpha_{t}'\geq\alpha_{\min}'=\frac{\delta_0}{8}\sqrt{\frac{\pi}{n}},\qquad \forall 1\leq t\leq\mathscr{T}'.
\end{align}
and it is demonstrated in the proof of \lem{AGD-component-lowerbound} that,
\begin{align}
\|\vect{x}_{t,\parallel}\|\geq\frac{\|\vect{x}_{0,\parallel}\|}{2}\cdot\Big(\frac{1+\kappa_{\parallel}}{2}\Big)^t\cdot(2-\theta+\mu_{\parallel})^t,\qquad\forall 1\leq t\leq\mathscr{T}',
\end{align}
for $\kappa_{\parallel}=\eta\sqrt{\rho\epsilon}-\eta\rho r'/\alpha_{\min}'$ and $\mu_{\parallel}=\sqrt{(2-\theta)^2-\frac{4(1-\theta)}{1+\kappa_{\parallel}}}$. Observe that
\begin{align}
\frac{\|\vect{x}_{\mathscr{T}',\perp'}\|}{\|\vect{x}_{\mathscr{T}',\parallel}\|}&\leq \frac{2\|\vect{x}_{0,\perp'}\|}{\|\vect{x}_{0,\parallel}\|}\cdot\Big(\frac{1+\kappa_{\perp'}}{1+\kappa_{\parallel}}\Big)^{\mathscr{T}'}\cdot\Big(\frac{2-\theta+\mu_{\perp'}}{2-\theta+\mu_{\parallel}}\Big)^{\mathscr{T}'}\\
&\leq\frac{2}{\delta_{0}}\sqrt{\frac{n}{\pi}}\Big(\frac{1+\kappa_{\perp'}}{1+\kappa_{\parallel}}\Big)^{\mathscr{T}'}\cdot\Big(\frac{2-\theta+\mu_{\perp'}}{2-\theta+\mu_{\parallel}}\Big)^{\mathscr{T}'},
\end{align}
where
\begin{align}
\frac{1+\kappa_{\perp'}}{1+\kappa_{\parallel}}
\leq \frac{1}{1+(\kappa_{\parallel}-\kappa_{\perp'})}
=1-\frac{1}{\eta\sqrt{\rho\epsilon}/2+\rho r'(\eta/\alpha_{\min'}+2/\sqrt{\rho\epsilon})}
\leq 1-\frac{\eta\sqrt{\rho\epsilon}}{4},
\end{align}
and
\begin{align}
\frac{2-\theta+\mu_{\perp'}}{2-\theta+\mu_{\parallel}}&=\frac{1+\sqrt{1-\frac{4(1-\theta)}{(1+\kappa_{\perp'})(2-\theta)^2}}}{1+\sqrt{1-\frac{4(1-\theta)}{(1+\kappa_{\parallel})(2-\theta)^2}}}\\
&\leq \frac{1}{1+\Big(\sqrt{1-\frac{4(1-\theta)}{(1+\kappa_{\perp'})(2-\theta)^2}}-\sqrt{1-\frac{4(1-\theta)}{(1+\kappa_{\parallel})(2-\theta)^2}}\Big)}\\
&\leq 1-\frac{\kappa_{\parallel}-\kappa_{\perp'}}{\theta}\\
&\leq 1-\frac{\eta\sqrt{\rho\epsilon}}{4\theta}=1-\frac{(\rho\epsilon)^{1/4}}{16\sqrt{\ell}}.
\end{align}
Hence,
\begin{align}
\frac{\|\vect{x}_{\mathscr{T}',\perp'}\|}{\|\vect{x}_{\mathscr{T}',\parallel}\|}&\leq\frac{2}{\delta_0}\sqrt{\frac{n}{\pi}}\Big(1-\frac{(\rho\epsilon)^{1/4}}{16\sqrt{\ell}}\Big)^{\mathscr{T}'}\leq\frac{\sqrt{\rho\epsilon}}{8\ell}.
\end{align}

Since $\|\vect{x}_{\mathscr{T}',\parallel}\|\leq\|\vect{x}_{\mathscr{T}'}\|$, we have $\frac{\|\vect{x}_{\mathscr{T}',\perp'}\|}{\|\vect{x}_{\mathscr{T}'}\|}\leq\frac{\sqrt{\rho\epsilon}}{8\ell}$, contradiction. Hence, there here exists some $t_0$ with $1\leq t_0\leq \mathscr{T}'$ such that $\frac{\|\vect{x}_{t_0,\perp'}\|}{\|\vect{x}_{t_0}\|}\leq\frac{\sqrt{\rho\epsilon}}{8\ell}$. Consider the normalized vector $\hat{\vect{e}}=\vect{x}_{t_0}/r$, we use $\hat{\vect{e}}_{\perp'}$ and $\hat{\vect{e}}_{\parallel'}$ to separately denote the component of $\hat{\vect{e}}$ in $\mathfrak{S}_{\perp}'$ and $\mathfrak{S}_{\parallel}'$. Then, $\|\hat{\vect{e}}_{\perp'}\|\leq\sqrt{\rho\epsilon}/(8\ell)$ whereas $\|\hat{\vect{e}}_{\parallel'}\|\geq 1-\rho\epsilon/(8\ell)^2$. Then,
\begin{align}
\hat{\vect{e}}^{T}\mathcal{H}(\vect{0})\hat{\vect{e}}=(\hat{\vect{e}}_{\perp'}+\hat{\vect{e}}_{\parallel'})^{T}\mathcal{H}(\vect{0})(\hat{\vect{e}}_{\perp'}+\hat{\vect{e}}_{\parallel'}),
\end{align}
since $\mathcal{H}(\vect{0})\hat{\vect{e}}_{\perp'}\in\mathfrak{S}_{\perp}'$ and $\mathcal{H}(\vect{0})\hat{\vect{e}}_{\parallel'}\in\mathfrak{S}_{\parallel}'$, it can be further simplified to
\begin{align}
\hat{\vect{e}}^{T}\mathcal{H}(\vect{0})\hat{\vect{e}}=\hat{\vect{e}}_{\perp'}^{T}\mathcal{H}(\vect{0})\hat{\vect{e}}_{\perp'}+\hat{\vect{e}}_{\parallel'}^{T}\mathcal{H}(\vect{0})\hat{\vect{e}}_{\parallel'},
\end{align}
Due to the $\ell$-smoothness of the function, all eigenvalue of the Hessian matrix has its absolute value upper bounded by $\ell$. Hence,
\begin{align}
\hat{\vect{e}}_{\perp'}^{T}\mathcal{H}(\vect{0})\hat{\vect{e}}_{\perp'}\leq\ell\|\hat{\vect{e}}_{\perp'}^{T}\|_{2}^2=\frac{\rho\epsilon}{64\ell^2}.
\end{align}
Further according to the definition of $\mathfrak{S}_{\parallel}$, we have
\begin{align}
\hat{\vect{e}}_{\parallel'}^{T}\mathcal{H}(\vect{0})\hat{\vect{e}}_{\parallel'}\leq-\frac{\sqrt{\rho\epsilon}}{2}\|\hat{\vect{e}}_{\parallel'}\|^2.
\end{align}
Combining these two inequalities together, we can obtain
\begin{align}
\hat{\vect{e}}^{T}\mathcal{H}(\vect{0})\hat{\vect{e}}&=\hat{\vect{e}}_{\perp}^{T}\mathcal{H}(\vect{0})\hat{\vect{e}}_{\perp'}+\hat{\vect{e}}_{\parallel'}^{T}\mathcal{H}(\vect{0})\hat{\vect{e}}_{\parallel'}
\leq-\frac{\sqrt{\rho\epsilon}}{2}\|\hat{\vect{e}}_{\parallel'}\|^2+\frac{\rho\epsilon}{64\ell^2}\leq-\frac{\sqrt{\rho\epsilon}}{4}.
\end{align}
\end{proof}

%%%%%%%%%%%%%%%%%%%%%%%%%%%%%%%%%%%%%%%%%%%%%%%%%%%%%%%

\section{Proof details of escaping from saddle points by negative curvature finding}

\subsection{Algorithms for escaping from saddle points using negative curvature finding}\label{append:explicit-algorithms}

In this subsection, we first present algorithm for escaping from saddle points using \algo{NC-finding} as \algo{PGD+NC}.
\begin{algorithm}[htbp]
\caption{Perturbed Gradient Descent with Negative Curvature Finding}
\label{algo:PGD+NC}
\textbf{Input:} $\vect{x}_{0}\in\mathbb{R}^n$\;
\For{$t=0,1,...,T$}{
\If{$\|\nabla f(\vect{x}_{t})\|\leq\epsilon$}{
$\hat{\vect{e}}\leftarrow$NegativeCurvatureFinding($\vect{x}_t,r,\mathscr{T}$)
\;
$\vect{x}_{t}\leftarrow \vect{x}_{t}-\frac{f'_{\hat{\vect{e}}}(\vect{x}_0)}{4|f'_{\hat{\vect{e}}}(\vect{x}_0)|}\sqrt{\frac{\epsilon}{\rho}}\cdot\hat{\vect{e}}$\;
}
$\vect{x}_{t+1}\leftarrow\vect{x}_{t}-\frac{1}{\ell}\nabla f(\vect{x}_{t})$\;
}
\end{algorithm}

Observe that \algo{PGD+NC} and \algo{ANCGD} are similar to perturbed gradient descent and perturbed accelerated gradient descent but the uniform perturbation step is replaced by our negative curvature finding algorithms. One may wonder that \algo{PGD+NC} seems to involve nested loops since negative curvature finding algorithm are contained in the primary loop, contradicting our previous claim that \algo{PGD+NC} only contains a single loop. But actually, \algo{PGD+NC} contains only two operations: gradient descents and one perturbation step, the same as operations outside the negative curvature finding algorithms. Hence, \algo{PGD+NC} is essentially single-loop algorithm, and we count their iteration number as the total number of gradient calls.

%============================================

\subsection{Proof details of escaping saddle points using \algo{NC-finding}}\label{append:PGD+NC}
In this subsection, we prove:
\begin{theorem}\label{thm:PGD+NC-Complexity}
For any $\epsilon>0$ and $0<\delta\leq 1$, \algo{PGD+NC} with parameters chosen in \prop{NC-finding} satisfies that at least $1/4$ of its iterations will be $\epsilon$-approximate second-order stationary point, using
\begin{align*}
\tilde{O}\Big(\frac{(f(\vect{x}_{0})-f^{*})}{\epsilon^{2}}\cdot\log n\Big)
\end{align*}
iterations, with probability at least $1-\delta$, where $f^{*}$ is the global minimum of $f$.
\end{theorem}

\begin{proof}
Let the parameters be chosen according to \eqn{parameter-choice}, and set the total step number $T$ to be:
\begin{align}
T=\max\left\{\frac{8\ell(f(\vect{x_{0}})-f^{*})}{\epsilon^{2}},768(f(\vect{x_{0}})-f^{*})\cdot\sqrt{\frac{\rho}{\epsilon^3}}\right\},
\end{align}
similar to the perturbed gradient descent algorithm {\cite[Algorithm 4]{jin2019stochastic}}. We first assume that for each $\vect{x}_t$ we apply negative curvature finding (\algo{NC-finding}) with $\delta_0$ contained in the parameters be chosen as
\begin{align}
\delta_0=\frac{1}{384(f(\vect{x}_0-f^{*})}\sqrt{\frac{\epsilon^3}{\rho}}\delta,
\end{align}
we can successfully obtain a unit vector $\hat{\vect{e}}$ with $\hat{\vect{e}}^{T}\mathcal{H}\hat{\vect{e}}\leq-\sqrt{\rho\epsilon}/4$, as long as $\lambda_{\min}(\mathcal{H}(\vect{x}_t))\leq-\sqrt{\rho\epsilon}$. The error probability of this assumption is provided later.

Under this assumption, \algo{NC-finding} can be called for at most $384(f(\vect{x_{0}})-f^{*})\sqrt{\frac{\rho}{\epsilon^3}}\leq \frac{T}{2}$ times, for otherwise the function value decrease will be greater than $f(\vect{x_{0}})-f^{*}$, which is not possible. Then, the error probability that some calls to \algo{NC-finding} fails is upper bounded by
\begin{align}
384(f(\vect{x_{0}})-f^{*})\sqrt{\frac{\rho}{\epsilon^3}}\cdot\delta_0=\delta.
\end{align}

For the rest of iterations in which \algo{NC-finding} is not called, they are either large gradient steps, $\|\nabla f(\vect{x}_{t})\|\geq \epsilon$, or $\epsilon$-approximate second-order stationary points. Within them, we know that the number of large gradient steps cannot be more than $T/4$ because otherwise, by \lem{descent-lemma} in \append{existing-lemma}:
\begin{align*}
f(\vect{x}_{T})\leq f(\vect{x}_{0})-T\eta\epsilon^{2}/8<f^{*},
\end{align*}
a contradiction. Therefore, we conclude that at least $T/4$ of the iterations must be $\epsilon$-approximate second-order stationary points, with probability at least $1-\delta$.

The number of iterations can be viewed as the sum of two parts, the number of iterations needed for gradient descent, denoted by $T_{1}$, and the number of iterations needed for negative curvature finding, denoted by $T_{2}$. with probability at least $1-\delta$,
\begin{align}
T_{1}=T=\tilde{O}\Big(\frac{(f(\vect{x}_{0})-f^{*})}{\epsilon^{2}}\Big).
\end{align}
As for $T_2$, with probability at least $1-\delta$, \algo{NC-finding} is called for at most $384(f(\vect{x_{0}})-f^{*})\sqrt{\frac{\rho}{\epsilon^3}}$ times, and by \prop{NC-finding} it takes $\tilde{O}\Big(\frac{\log n}{\sqrt{\rho\epsilon}}\Big)$ iterations each time. Hence,
\begin{align}
T_2=384(f(\vect{x_{0}})-f^{*})\sqrt{\frac{\rho}{\epsilon^3}}\cdot \tilde{O}\Big(\frac{\log n}{\sqrt{\rho\epsilon}}\Big)=\tilde{O}\Big(\frac{(f(\vect{x}_{0})-f^{*})}{\epsilon^{2}}\cdot\log n\Big).
\end{align}
As a result, the total iteration number $T_1+T_2$ is
\begin{align}
\tilde{O}\Big(\frac{(f(\vect{x}_{0})-f^{*})}{\epsilon^{2}}\cdot\log n\Big).
\end{align}
\end{proof}

%==================================================

\subsection{Proof details of escaping saddle points using \algo{ANCGD}}
\label{append:PAGD+ANC}
We first present here the Negative Curvature Exploitation algorithm proposed in proposed in {\cite[Algorithm 3]{jin2018accelerated}} appearing in \lin{NCE} of \algo{ANCGD}:
\begin{algorithm}[htbp]
\caption{Negative Curvature Exploitation($\vect{x}_t,\vect{v}_t,s)$}
\label{algo:NCE}
\If{$\|\vect{v}_t\|\geq s$}{
$\vect{x}_{t+1}\leftarrow\vect{x}_t$\;
}
\Else{
$\xi=s\cdot \vect{v}_t/\|\vect{v}\|_t$\;
$\vect{x}_t\leftarrow\text{argmin}_{\vect{x}\in \left\{\vect{x}_t+\xi,\vect{x}_t-\xi\right\} }f(\vect{x})$\;
}
\textbf{Output} $(\vect{z}_{t+1},\vect{0})$.
\end{algorithm}

Now, we give the full version of \thm{PAGD+ANC-Complexity} as follows:
\begin{theorem}\label{thm:PAGD+ANC-full}
Suppose that the function $f$ is $\ell$-smooth and $\rho$-Hessian Lipschitz. For any $\epsilon>0$ and a constant $0<\delta\leq 1$, we choose the parameters appearing in \algo{ANCGD} as follows:
\begin{align}\label{eqn:ANCGD-parameter-choice}
\delta_{0}&=\frac{\delta}{384\Delta_f}\sqrt{\frac{\epsilon^3}{\rho}},
&\mathscr{T}'&=\frac{32\sqrt{\ell}}{(\rho\epsilon)^{1/4}}\log\Big(\frac{\ell}{\delta_0}\sqrt{\frac{n}{\rho\epsilon}}\Big),
&\zeta &=\frac{\ell}{\sqrt{\rho\epsilon}},\\
r'&=\frac{\delta_0\epsilon}{32}\sqrt{\frac{\pi}{\rho n}}
,
&\eta &=\frac{1}{4\ell},
&\theta &=\frac{1}{4\sqrt{\zeta}},
\\
\mathscr{E} &=\sqrt{\frac{\epsilon^{3}}{\rho}}\cdot c_{A}^{-7},
&\gamma &=\frac{\theta^2}{\eta},
&s &=\frac{\gamma}{4\rho},
\end{align}
where $\Delta_f:=f(\vect{x}_0)-f^{*}$ and $f^*$ is the global minimum of $f$, and the constant $c_A$ is chosen large enough to satisfy both the condition in \lem{AGD-large-gradient} and $c_A\geq(384)^{1/7}$. Then, \algo{ANCGD} satisfies that at least one of the iterations $\vect{z}_{t}$ will be an $\epsilon$-approximate second-order stationary point in
\begin{align}
\tilde{O}\Big(\frac{(f(\vect{x}_{0})-f^{*})}{\epsilon^{1.75}}\cdot\log n\Big)
\end{align}
iterations, with probability at least $1-\delta$.
\end{theorem}
\begin{proof}
Set the total step number $T$ to be:
\begin{align}
T=\max\left\{\frac{4\Delta_f(\tilde{\mathscr{T}}+\mathscr{T}')}{\mathscr{E}},768\Delta_f\mathscr{T}'\sqrt{\frac{\rho}{\epsilon^3}}\right\}=\tilde{O}\Big(\frac{(f(\vect{x}_{0})-f^{*})}{\epsilon^{1.75}}\cdot\log n\Big),
\end{align}
where $\tilde{\mathscr{T}}=\sqrt{\zeta}\cdot c_A$ as defined in \lem{AGD-large-gradient}, similar to the perturbed accelerated gradient descent algorithm {\cite[Algorithm 2]{jin2018accelerated}}. We first assert that for each iteration $\vect{x}_t$ that a uniform perturbation is added, after $\mathscr{T}'$ iterations we can successfully obtain a unit vector $\hat{\vect{e}}$ with $\hat{\vect{e}}^{T}\mathcal{H}\hat{\vect{e}}\leq-\sqrt{\rho\epsilon}/4$, as long as $\lambda_{\min}(\mathcal{H}(\vect{x}_t))\leq-\sqrt{\rho\epsilon}$. The error probability of this assumption is provided later.

Under this assumption, the uniform perturbation can be called for at most $384(f(\vect{x_{0}})-f^{*})\sqrt{\frac{\rho}{\epsilon^3}}$ times, for otherwise the function value decrease will be greater than $f(\vect{x_{0}})-f^{*}$, which is not possible. Then, the probability that at least one negative curvature finding subroutine after uniform perturbation fails is upper bounded by
\begin{align}
384(f(\vect{x_{0}})-f^{*})\sqrt{\frac{\rho}{\epsilon^3}}\cdot\delta_0=\delta.
\end{align}

For the rest of steps which is not within $\mathscr{T}'$ steps after uniform perturbation, they are either large gradient steps, $\|\nabla f(\vect{x}_{t})\|\geq \epsilon$, or $\epsilon$-approximate second-order stationary points. Next, we demonstrate that at least one of these steps is an $\epsilon$-approximate stationary point.

Assume the contrary. We use $N_{\tilde{\mathscr{T}}}$ to denote the number of disjoint time periods with length larger than $\tilde{\mathscr{T}}$ containing only large gradient steps and do not contain any step within $\mathscr{T}'$ steps after uniform perturbation. Then, it satisfies
\begin{align}
N_{\tilde{\mathscr{T}}}\geq\frac{T}{2(\tilde{\mathscr{T}}+\mathscr{T}')}-384\Delta_f\sqrt{\frac{\rho}{\epsilon^3}}\geq(2c_A^7-384)\Delta_f\sqrt{\frac{\rho}{\epsilon^3}}\geq \frac{\Delta_f}{\mathscr{E}}.
\end{align}
From \lem{AGD-large-gradient}, during these time intervals the Hamiltonian will decrease in total at least $N_{\tilde{\mathscr{T}}}\cdot\mathscr{E}=\Delta_f$, which is impossible due to \lem{Hamiltonian-decrease}, the Hamiltonian decreases monotonically for every step except for the $\mathscr{T}'$ steps after uniform perturbation, and the overall decrease cannot be greater than $\Delta_f$, a contradiction. Therefore, we conclude that at least one of the iterations must be an $\epsilon$-approximate second-order stationary point, with probability at least $1-\delta$.
\end{proof}

%%%%%%%%%%%%%%%%%%%%%%%%%%%%%%%%%%%%%%%%%%%%%%%%%%%%%%

\section{Proofs of the stochastic setting}\label{append:PSGD+SNC}
\subsection{Proof details of negative curvature finding using stochastic gradients}\label{append:SNC-finding}
In this subsection, we demonstrate that \algo{SNC-finding} can find a negative curvature efficiently. Specifically, we prove the following proposition:
\begin{proposition}\label{prop:SNC-finding}
Suppose the function $f\colon\mathbb{R}^n\to\mathbb{R}$ is $\ell$-smooth and $\rho$-Hessian Lipschitz. For any $0<\delta<1$, we specify our choice of parameters and constants we use as follows:
\begin{align}\label{eqn:stochastic-parameter-choice}
\mathscr{T}_s&=\frac{8\ell}{\sqrt{\rho\epsilon}}\cdot\log\Big(\frac{\ell n}{\delta\sqrt{\rho\epsilon}}\Big),
&\iota &=10\log\Big(\frac{n\mathscr{T}_s^2}{\delta}\log\Big(\frac{\sqrt{n}}{\eta r_s}\Big)\Big),\\
r_s&=\frac{\delta}{480\rho n\mathscr{T}_s}\sqrt{\frac{\rho\epsilon}{\iota}},& m&=\frac{160(\ell+\tilde{\ell})}{\delta\sqrt{\rho\epsilon}}\sqrt{\mathscr{T}_s\iota},
\end{align}
Then for any point $\tilde{\vect{x}}\in\mathbb{R}^{n}$ satisfying $\lambda_{\min}(\mathcal{H}(\tilde{\vect{x}}))\leq-\sqrt{\rho\epsilon}$, with probability at least $1-3\delta$, \algo{SNC-finding} outputs a unit vector $\hat{\vect{e}}$ satisfying
\begin{align}
\hat{\vect{e}}^{T}\mathcal{H}(\tilde{\vect{x}})\hat{\vect{e}}\leq-\frac{\sqrt{\rho\epsilon}}{4},
\end{align}
where $\mathcal{H}$ stands for the Hessian matrix of function $f$, using $O(m\cdot\mathscr{T}_s)=\tilde{O}\Big(\frac{\log^2 n}{\delta\epsilon^{1/2}}\Big)$ iteartions.
\end{proposition}

Similarly to \algo{NC-finding} and \algo{ANCGD}, the renormalization step \lin{y_t-renormalization-stochastic} in \algo{SNC-finding} only guarantees that the value $\|\vect{y}_t\|$ would not scales exponentially during the algorithm, and does not affect the output. We thus introduce the following \algo{SNC-finding-WN}, which is the no-renormalization version of \algo{SNC-finding} that possess the same output and a simpler structure. Hence in this subsection, we analyze \algo{SNC-finding-WN} instead of \algo{SNC-finding}.
\begin{algorithm}[htbp]
\caption{Stochastic Negative Curvature Finding without Renormalization($\tilde{\vect{x}},r_s,\mathscr{T}_s,m$).}
\label{algo:SNC-finding-WN}
$\vect{z}_0\leftarrow0$\;
\For{$t=1,...,\mathscr{T}_s$}{
Sample $\left\{\theta^{(1)},\theta^{(2)},\cdots,\theta^{(m)}\right\}\sim\mathcal{D}$\;
$\vect{g}(\vect{z}_{t-1})\leftarrow\frac{\|\vect{z}_{t-1}\|}{r_s}\cdot\frac{1}{m}\sum_{j=1}^m\Big(\vect{g}\Big(\tilde{\vect{x}}+\frac{r_s}{\|\vect{z}_{t-1}\|}\vect{z}_{t-1};\theta^{(j)}\Big)-\vect{g}(\tilde{\vect{x}};\theta^{(j)})\Big)$\;
$\vect{z}_t\leftarrow\vect{z}_{t-1}-\frac{1}{\ell}(\vect{g}(\vect{z}_{t-1})+\xi_t),\qquad\xi_t\sim\mathcal{N}\Big(0,\frac{r_s^2}{d}I\Big)$\label{lin:z_t-definition}\;
}
\textbf{Output} $\vect{z}_{\mathscr{T}}/\|\vect{z}_{\mathscr{T}}\|$.
\end{algorithm}

Without loss of generality, we assume $\tilde{\vect{x}}=\vect{0}$ by shifting $\mathbb{R}^n$ such that $\tilde{\vect{x}}$ is mapped to $\vect{0}$. As argued in the proof of \prop{NC-finding}, $\mathcal{H}(\vect{0})$ admits the following following eigen-decomposition:
\begin{align}
\mathcal{H}(\vect{0})=\sum_{i=1}^{n}\lambda_i\vect{u}_i\vect{u}_i^T,
\end{align}
where the set $\{\vect{u}_i\}_{i=1}^{n}$ forms an orthonormal basis of $\mathbb{R}^n$. Without loss of generality, we assume the eigenvalues $\lambda_1,\lambda_2,\ldots,\lambda_n$ corresponding to $\vect{u}_1,\vect{u}_2,\ldots,\vect{u}_n$ satisfy
\begin{align}
\lambda_1\leq\lambda_2\leq\cdots\leq\lambda_n,
\end{align}
where $\lambda_1\leq-\sqrt{\rho\epsilon}$. If $\lambda_n\leq-\sqrt{\rho\epsilon}/2$, \prop{SNC-finding} holds directly. Hence, we only need to prove the case where $\lambda_n>-\sqrt{\rho\epsilon}/2$, where there exists some $p>1$ and $p'>1$ with
\begin{align}
\lambda_p\leq -\sqrt{\rho\epsilon}\leq \lambda_{p+1},\quad\lambda_{p'}\leq -\sqrt{\rho\epsilon}/2< \lambda_{p'+1}.
\end{align}

\paragraph{Notation:}Throughout this subsection, let $\tilde{\mathcal{H}}:=\mathcal{H}(\tilde{\vect{x}})$. Use $\mathfrak{S}_{\parallel}$, $\mathfrak{S}_{\perp}$ to separately denote the subspace of $\mathbb{R}^{n}$ spanned by $\left\{\vect{u}_1,\vect{u}_2,\ldots,\vect{u}_p\right\}$, $\left\{\vect{u}_{p+1},\vect{u}_{p+2},\ldots,\vect{u}_{n}\right\}$, and use $\mathfrak{S}_{\parallel}'$, $\mathfrak{S}_{\perp}'$ to denote the subspace of $\mathbb{R}^{n}$ spanned by $\left\{\vect{u}_1,\vect{u}_2,\ldots,\vect{u}_{p'}\right\}$, $\left\{\vect{u}_{p'+1},\vect{u}_{p+2},\ldots,\vect{u}_{n}\right\}$. Furthermore, define $\vect{z}_{t,\parallel}:=\sum_{i=1}^p\left\<\vect{u}_i,\vect{z}_t\right\>\vect{u}_i$,
$\vect{z}_{t,\perp}:=\sum_{i=p}^n\left\<\vect{u}_i,\vect{z}_t\right\>\vect{u}_i$,
$\vect{z}_{t,\parallel'}:=\sum_{i=1}^{p'}\left\<\vect{u}_i,\vect{z}_t\right\>\vect{u}_i$,
$\vect{z}_{t,\perp'}:=\sum_{i=p'}^n\left\<\vect{u}_i,\vect{z}_t\right\>\vect{u}_i$
respectively to denote the component of $\vect{z}_t$ in \lin{z_t-definition} of \algo{SNC-finding-WN} in the subspaces $\mathfrak{S}_{\parallel}$, $\mathfrak{S}_{\perp}$, $\mathfrak{S}_{\parallel}'$, $\mathfrak{S}_{\perp}'$, and let $\gamma=-\lambda_1$.

To prove \prop{SNC-finding}, we first introduce the following lemma:

\begin{lemma}\label{lem:SNC-overlap}
Under the setting of \prop{SNC-finding}, for any point $\tilde{\vect{x}}\in\mathbb{R}^{n}$ satisfying $\lambda_{\min}(\nabla^2f(\tilde{\vect{x}})\leq-\sqrt{\rho\epsilon}$, with probability at least $1-3\delta$, \algo{SNC-finding} outputs a unit vector $\hat{\vect{e}}$ satisfying
\begin{align}
\|\hat{\vect{e}}_{\perp'}\|
:=\Big\|\sum_{i=p'}^n\left\<\vect{u}_i,\hat{\vect{e}}\right\>\vect{u}_i\Big\|\leq\frac{\sqrt{\rho\epsilon}}{8\ell}
\end{align}
using $O(m\cdot\mathscr{T}_s)=\tilde{O}\Big(\frac{\log^2 n}{\delta\epsilon^{1/2}}\Big)$ iteartions.
\end{lemma}

\subsubsection{Proof of \lem{SNC-overlap}}
In the proof of \lem{SNC-overlap}, we consider the worst case, where $\lambda_1=-\gamma=-\sqrt{\rho\epsilon}$ is the only eigenvalue less than $-\sqrt{\rho\epsilon}/2$, and all other eigenvalues equal to $-\sqrt{\rho\epsilon}/2+\nu$ for an arbitrarily small constant $\nu$. Under this scenario, the component $\vect{z}_{t,\perp'}$ is as small as possible at each time step.

The following lemma characterizes the dynamics of \algo{SNC-finding-WN}:

\begin{lemma}\label{lem:CS-dynamics}
Consider the sequence $\left\{\vect{z}_i\right\}$ and let $\eta=1/\ell$. Further, for any $0\leq t\leq \mathscr{T}_s$ we define
\begin{align}
\zeta_{t}:=\vect{g}(\vect{z}_{t-1})-\frac{\|\vect{z}_{t}\|}{r_s}\Big(\nabla f\Big(\tilde{\vect{x}}+\frac{r_s}{\|\vect{z}_t\|}\vect{z}_t\Big)-\nabla f(\tilde{\vect{x}})\Big),
\end{align}
to be the errors caused by the stochastic gradients. Then $\vect{z}_t=-\vect{q}_h(t)-\vect{q}_{sg}(t)-\vect{q}_p(t)$, where:
\begin{align}
\vect{q}_h(t):=\eta\sum_{\tau=0}^{t-1}(I-\eta\tilde{\mathcal{H}})^{t-1-\tau}\Delta_{\tau}\hat{\vect{z}}_{\tau},
\end{align}
for $\Delta_{\tau}=\int_{0}^{1}\mathcal{H}_f\big(\psi\frac{r_s}{\|\vect{z}_{\tau}\|}\vect{z}_{\tau}\big)\d \psi-\tilde{\mathcal{H}}$, and
\begin{align}
\vect{q}_{sg}(t):=\eta\sum_{\tau=0}^{t-1}(I-\eta\tilde{\mathcal{H}})^{t-1-\tau}\zeta_{\tau},\quad\vect{q}_{p}(t):=\eta\sum_{\tau=0}^{t-1}(I-\eta\tilde{\mathcal{H}})^{t-1-\tau}\xi_{\tau}.
\end{align}
\end{lemma}
\begin{proof}
Without loss of generality we assume $\tilde{\vect{x}}=\vect{0}$. The update formula for $\vect{z}_t$ can be written as
\begin{align}
\vect{z}_{t+1}=\vect{z}_t-\eta\Big(\frac{\|\vect{z}_{t}\|}{r_s}\Big(\nabla f\Big(\frac{r_s}{\|\vect{z}_t\|}\vect{z}_t\Big)-\nabla f(\vect{0})\Big)+\zeta_t+\xi_t\Big),
\end{align}
where
\begin{align}
\frac{\|\vect{z}_{t}\|}{r_s}\Big(\nabla f\Big(\frac{r_s}{\|\vect{z}_t\|}\vect{z}_t\Big)-\nabla f(\vect{0})\Big)
=\frac{\|\vect{z}_{t}\|}{r_s}\int_{0}^{1}\mathcal{H}_f\Big(\psi\frac{r_s}{\|\vect{z}_t\|}\vect{z}_t\Big)\frac{r_s}{\|\vect{z}_t\|}\vect{z}_t\d\psi=(\tilde{\mathcal{H}}+\Delta_t)\vect{z}_t,
\end{align}
indicating
\begin{align}
\vect{z}_{t+1}&=(I-\eta\tilde{\mathcal{H}})\vect{x}_t-\eta(\Delta_t\vect{z}_t+\zeta_t+\xi_t)\\
&=-\eta\sum_{\tau=0}^t(I-\eta\tilde{\mathcal{H}})^{t-\tau}(\Delta_t\vect{z}_t+\zeta_t+\xi_t),
\end{align}
which finishes the proof.
\end{proof}
At a high level, under our parameter choice in \prop{SNC-finding}, $\vect{q}_p(t)$ is the dominating term controlling the dynamics, and $\vect{q}_h(t)+\vect{q}_{sg}(t)$ will be small compared to $\vect{q}_p(t)$. Quantitatively, this is shown in the following lemma:
\begin{lemma}\label{lem:qp-dominating}
Under the setting of \prop{SNC-finding} while using the notation in \lem{stochastic-ab} and \lem{CS-dynamics}, we have
\begin{align}
\Pr\Big(\|\vect{q}_h(t)+\vect{q}_{sg}(t)\|\leq\frac{\beta(t)\eta r_s\delta}{20\sqrt{n}}\cdot\frac{\sqrt{\rho\epsilon}}{16\ell},\ \forall t\leq\mathscr{T}_s\Big)\geq 1-\delta,
\end{align}
where $-\gamma:=\lambda_{\min}(\tilde{\mathcal{H}})=-\sqrt{\rho\epsilon}$.
\end{lemma}
\begin{proof}
Divide $\vect{q}_p(t)$ into two parts:
\begin{align}
\vect{q}_{p,1}(t):=\left\<\vect{q}_{p}(t),\vect{u}_1\right\>\vect{u}_1,
\end{align}
and
\begin{align}
\vect{q}_{p,\perp'}(t):=\vect{q}_p(t)-\vect{q}_{p,1}(t).
\end{align}
Then by \lem{stochastic-qp-norm}, we have
\begin{align}
\Pr\Big(\|\vect{q}_{p,1}(t)\|\leq\frac{\beta(t)\eta r_s}{\sqrt{n}}\cdot\sqrt{\iota}\Big)\geq 1-2e^{-\iota},
\end{align}
and
\begin{align}
\Pr\Big(\|\vect{q}_{p,1}(t)\|\geq\frac{\beta(t)\eta r_s}{20\sqrt{n}}\cdot\delta\Big)\geq 1-\delta/4.
\end{align}
Similarly,
\begin{align}
\Pr\Big(\|\vect{q}_{p,\perp'}(t)\|\leq\beta_{\perp'}(t)\eta r_s\cdot\sqrt{\iota}\Big)\geq 1-2e^{-\iota},
\end{align}
and
\begin{align}
\Pr\Big(\|\vect{q}_{p,\perp'}(t)\|\geq\frac{\beta_{\perp'}(t)\eta r_s}{20}\cdot\delta\Big)\geq 1-\delta/4,
\end{align}
where $\beta_{\perp'}(t):=\frac{(1+\eta\gamma/2)^t}{\sqrt{\eta\gamma}}$. Set $t_{\perp'}:=\frac{\log n}{\eta\gamma}$. Then for all $\tau\leq t_{\perp'}$, we have
\begin{align}
\frac{\beta(\tau)}{\beta_{\perp'}(\tau)}\leq\sqrt{n},
\end{align}
which further leads to
\begin{align}
\Pr\Big(\|\vect{q}_{p,\perp'}(\tau)\|\leq 2\beta_{\perp'}(t)\eta r_s\cdot\sqrt{\iota}\Big)\geq 1-2e^{-\iota}.
\end{align}
Next, we use induction to prove that the following inequality holds for all $t\leq t_{\perp'}$:
\begin{align}\label{eqn:first-recurrence}
\Pr\Big(\|\vect{q}_h(\tau)+\vect{q}_{sg}(\tau)\|\leq\beta_{\perp'}(\tau)\eta r_s\cdot\frac{\delta}{20},\ \forall \tau\leq t\Big)\geq 1-10n t^2\log\Big(\frac{\sqrt{n}}{\eta r_s}\Big)e^{-\iota}.
\end{align}
For the base case $t=0$, the claim holds trivially. Suppose it holds for all $\tau\leq t$ for some $t$. Then due to \lem{stochastic-qp-norm}, with probability at least $1-2t_{\perp'} e^{-\iota}$, we have
\begin{align}
\|\vect{z}_{t}\|\leq \eta\|\vect{q}_p(t)\|+\eta\|\vect{q}_h(t)+\vect{q}_{sg}(t)\|\leq 3\beta_{\perp'}(\tau)\eta r_s\cdot\sqrt{\iota}.
\end{align}
By the Hessian Lipschitz property, $\Delta_{\tau}$ satisfies:
\begin{align}
\|\Delta_{\tau}\|\leq\rho r_s.
\end{align}
Hence,
\begin{align}
\|\vect{q}_h(t+1)\|&\leq\big\|\eta\sum_{\tau=0}^{t}(I-\eta\tilde{\mathcal{H}})^{t-\tau}\Delta_{\tau}\vect{z}_{\tau}\big\|\\
&\leq \eta\rho r_s\sum_{\tau=0}^{t}(I-\eta\tilde{\mathcal{H}})^{t-\tau}\|\vect{z}_{\tau}\|\\
&\leq(\eta\rho r_s n\mathscr{T}_s)\cdot(3\beta_{\perp'}(t)\eta r_s)\cdot\sqrt{\iota}\\
&\leq\frac{\beta_{\perp'}(t+1)\eta r_s}{10\sqrt{n}}\cdot\frac{\delta\sqrt{\rho\epsilon}}{16\ell}.
\end{align}
As for $\vect{q}_{sg}(t)$, note that $\hat{\zeta}_{\tau}|\mathcal{F}_{\tau-1}$ satisfies the norm-subGaussian property defined in \defn{nSG}. Specifically, $\hat{\zeta}_{\tau}|\mathcal{F}_{\tau-1}\sim \text{nSG}((\ell+\tilde{\ell})\|\hat{\vect{z}}_{\tau}\|/\sqrt{m})$. By applying \lem{stochastic-contraction-37} with $b=\alpha^2(t)\cdot\eta^2(\ell+\tilde{\ell})^2/m$ and $b=\alpha^2(t)\eta^2(\ell+\tilde{\ell})^2\eta^2 r_s^2/(mn)$, with probability at least
\begin{align}
1-4n\cdot\log\Big(\frac{\sqrt{n}}{\eta r_s}\Big)\cdot e^{-\iota},
\end{align}
we have
\begin{align}
\|\vect{q}_{sg}(t+1)\|\leq\frac{\eta(\ell+\tilde{\ell})\sqrt{t}}{m}\cdot(\beta_{\perp}(t)\eta r_s)\cdot\sqrt{\iota}\leq\frac{\beta_{\perp'}(t+1)\eta r_s}{20}\cdot\frac{\delta\sqrt{\rho\epsilon}}{8\ell}.
\end{align}
Then by union bound, with probability at least
\begin{align}
1-10n (t+1)^2\log\Big(\frac{\sqrt{n}}{\eta r_s}\Big)e^{-\iota},
\end{align}
we have
\begin{align}
\|\vect{q}_h(t+1)+\vect{q}_{sg}(t+1)\|\leq\beta_{\perp'}(t+1)\eta r_s\cdot\frac{\delta}{20}\cdot\frac{\sqrt{\rho\epsilon}}{8\ell},
\end{align}
indicating that \eqn{first-recurrence} holds. Then with probability at least
\begin{align}
1-10n t_{\perp'}^2\log\Big(\frac{\sqrt{n}}{\eta r_s}\Big)e^{-\iota}-\delta/4,
\end{align}
we have
\begin{align}
\|\vect{q}_h(t_{\perp'})+\vect{q}_{sg}(t_{\perp'})\|\leq\|\vect{q}_{p,1}(t_{\perp'})\|\cdot\frac{\sqrt{\rho\epsilon}}{16\ell}.
\end{align}
Based on this, we prove that the following inequality holds for any $t_{\perp'}\leq t\leq\mathscr{T}_s$:
\begin{align}\label{eqn:second-recurrence}
\Pr\Big(\|\vect{q}_h(\tau)+\vect{q}_{sg}(\tau)\|\leq\frac{\beta(\tau)\eta r_s}{20\sqrt{n}}\cdot\frac{\delta\sqrt{\rho\epsilon}}{16\ell},\ \forall t_{\perp'}\leq\tau\leq t\Big)\geq 1-10n t^2\log\Big(\frac{\sqrt{n}}{\eta r_s}\Big)e^{-\iota}.
\end{align}
We still use recurrence to prove it. Note that its base case $\tau=t_{\perp'}$ is guaranteed by \eqn{first-recurrence}. Suppose it holds for all $\tau\leq t$ for some $t$. Then with probability at least $1-2t e^{-\iota}$, we have
\begin{align}
\|\vect{z}_{t}\|
&\leq \eta\|\vect{q}_p(t)\|+\eta\|\vect{q}_h(t)+\vect{q}_{sg}(t)\|\\
&\leq 2\|\vect{q}_{p,1}(t)\|+\eta\|\vect{q}_h(t)+\vect{q}_{sg}(t)\|\\
&\leq \frac{3\beta(\tau)\eta r_s}{\sqrt{n}}\cdot\sqrt{\iota}.
\end{align}
Then following a similar procedure as before, we can claim that
\begin{align}
\|\vect{q}_h(t+1)+\vect{q}_{sg}(t+1)\|\leq\frac{\beta(t+1)\eta r_s}{\sqrt{n}}\cdot\frac{\delta}{20}\cdot\frac{\sqrt{\rho\epsilon}}{8\ell},
\end{align}
holds with probability
\begin{align}
1-10n (t+1)^2\log\Big(\frac{\sqrt{n}}{\eta r_s}\Big)e^{-\iota}-\frac{\delta}{4},
\end{align}
indicating that \eqn{second-recurrence} holds. Then under our choice of parameters, the desired inequality
\begin{align}
\|\vect{q}_h(t)+\vect{q}_{sg}(t)\|\leq\frac{\beta(t)\eta r_s\delta}{20\sqrt{n}}\cdot\frac{\sqrt{\rho\epsilon}}{16\ell}
\end{align}
holds with probability at least $1-\delta$.
\end{proof}

Equipped with \lem{qp-dominating}, we are now ready to prove \lem{SNC-overlap}.
\begin{proof}
First note that under our choice of $\mathscr{T}_s$, we have
\begin{align}
\Pr\Big(\frac{\|\vect{q}_{p,\perp'}(\mathscr{T}_s)\|}{\|\vect{q}_{p,1}(\mathscr{T}_s)\|}\leq\frac{\sqrt{\rho\epsilon}}{16\ell}\Big)\geq 1-\delta.
\end{align}
Further by \lem{qp-dominating} and union bound, with probability at least $1-2\delta$,
\begin{align}
\frac{\|\vect{q}_h(\mathscr{T}_s)+\vect{q}_{sg}(\mathscr{T}_s)\|}{\|\vect{q}_p(\mathscr{T}_s)\|}
\leq \|\vect{q}_h(\mathscr{T}_s)+\vect{q}_{sg}(\mathscr{T}_s)\|\cdot \frac{20\sqrt{n}}{\delta \beta(t)\eta r_s}
\leq \frac{\sqrt{\rho\epsilon}}{16\ell}.
\end{align}
For the output $\hat{\vect{e}}$, observe that its component $\hat{\vect{e}}_{\perp'}=\hat{\vect{e}}-\hat{\vect{e}}_1$, since $\vect{u}_1$ is the only component in subspace $\mathfrak{S}_{\parallel'}$.
Then with probability at least $1-3\delta$,
\begin{align}
\|\hat{\vect{e}}_{\perp'}\|\leq\sqrt{\rho\epsilon}/(8\ell).
\end{align}
\end{proof}

\subsubsection{Proof of \prop{SNC-finding}}
Based on \lem{SNC-overlap}, we present the proof of \prop{SNC-finding} as follows:
\begin{proof}
By \lem{SNC-overlap}, the component $\hat{\vect{e}}_{\perp'}$ of output $\vect{e}$ satisfies
\begin{align}
\|\hat{\vect{e}}_{\perp'}\|\leq\frac{\sqrt{\rho\epsilon}}{8\ell}.
\end{align}
Since $\hat{\vect{e}}=\hat{\vect{e}}_{\parallel'}+\hat{\vect{e}}_{\perp'}$, we can derive that
\begin{align}
\|\hat{\vect{e}}_{\parallel'}\|\geq\sqrt{1-\frac{\rho\epsilon}{(8\ell)^2}}\geq 1-\frac{\rho\epsilon}{(8\ell)^2}.
\end{align}
Note that
\begin{align}
\hat{\vect{e}}^{T}\tilde{\mathcal{H}}\hat{\vect{e}}=(\hat{\vect{e}}_{\perp'}+\hat{\vect{e}}_{\parallel'})^{T}\tilde{\mathcal{H}}(\hat{\vect{e}}_{\perp'}+\hat{\vect{e}}_{\parallel'}),
\end{align}
which can be further simplified to
\begin{align}
\hat{\vect{e}}^{T}\tilde{\mathcal{H}}\hat{\vect{e}}=\hat{\vect{e}}_{\perp'}^{T}\tilde{\mathcal{H}}\hat{\vect{e}}_{\perp'}+\hat{\vect{e}}_{\parallel'}^{T}\tilde{\mathcal{H}}\hat{\vect{e}}_{\parallel'}.
\end{align}
Due to the $\ell$-smoothness of the function, all eigenvalue of the Hessian matrix has its absolute value upper bounded by $\ell$. Hence,
\begin{align}
\hat{\vect{e}}_{\perp}^{T}\tilde{\mathcal{H}}\hat{\vect{e}}_{\perp}\leq\ell\|\hat{\vect{e}}_{\perp}^{T}\|_{2}^2=\frac{\rho\epsilon}{64\ell^2},
\end{align}
whereas
\begin{align}
\hat{\vect{e}}_{\parallel'}^{T}\tilde{\mathcal{H}}\hat{\vect{e}}_{\parallel'}\leq-\frac{\sqrt{\rho\epsilon}}{2}\|\hat{\vect{e}}_{\parallel'}\|^2.
\end{align}
Combining these two inequalities together, we can obtain
\begin{align}
\hat{\vect{e}}^{T}\tilde{\mathcal{H}}\hat{\vect{e}}&=\hat{\vect{e}}_{\perp'}^{T}\tilde{\mathcal{H}}\hat{\vect{e}}_{\perp'}+\hat{\vect{e}}_{\parallel'}^{T}\tilde{\mathcal{H}}\hat{\vect{e}}_{\parallel'}
\leq-\frac{\sqrt{\rho\epsilon}}{2}\|\hat{\vect{e}}_{\parallel'}\|^2+\frac{\rho\epsilon}{64\ell^2}\leq-\frac{\sqrt{\rho\epsilon}}{4}.
\end{align}
\end{proof}

%=========================================
\subsection{Proof details of escaping saddle points using \algo{SNC-finding}}\label{append:PSGD+SNC}
In this subsection, we demonstrate that \algo{SNC-finding} can be used to escape from saddle points in the stochastic setting. We first present the explicit \algo{SGD+NC}, and then introduce the full version \thm{PSGD+SNC-Complexity} with proof.
\begin{algorithm}[htbp]
\caption{Stochastic Gradient Descent with Negative Curvature Finding.}
\label{algo:SGD+NC}
\textbf{Input:} $\vect{x}_{0}\in\mathbb{R}^n$\;
\For{$t=0,1,...,T$}{
Sample $\left\{\theta^{(1)},\theta^{(2)},\cdots,\theta^{(M)}\right\}\sim\mathcal{D}$\;
$\vect{g}(\vect{x}_{t})=\frac{1}{M}\sum_{j=1}^M\vect{g}(\vect{x}_t;\theta^{(j)})$\;
\If{$\|\vect{g}(\vect{x}_{t})\|\leq3\epsilon/4$}{
$\hat{\vect{e}}\leftarrow$StochasticNegativeCurvatureFinding($\vect{x}_t,r_s,\mathscr{T}_s,m$)\;
$\vect{x}_{t}\leftarrow \vect{x}_{t}-\frac{f'_{\hat{\vect{e}}}(\vect{x}_0)}{4|f'_{\hat{\vect{e}}}(\vect{x}_0)|}\sqrt{\frac{\epsilon}{\rho}}\cdot\hat{\vect{e}}$\;
Sample $\left\{\theta^{(1)},\theta^{(2)},\cdots,\theta^{(M)}\right\}\sim\mathcal{D}$\;
$\vect{g}(\vect{x}_{t})=\frac{1}{M}\sum_{j=1}^M\vect{g}(\vect{x}_t;\theta^{(j)})$\;
}
$\vect{x}_{t+1}\leftarrow\vect{x}_{t}-\frac{1}{\ell}\vect{g}(\vect{x}_t;\theta_t)$\;
}
\end{algorithm}
\begin{theorem}[Full version of \thm{PSGD+SNC-Complexity}]\label{thm:PSGD+SNC-full}
Suppose that the function $f$ is $\ell$-smooth and $\rho$-Hessian Lipschitz. For any $\epsilon>0$ and a constant $0<\delta_s\leq 1$, we choose the parameters appearing in \algo{SGD+NC} as
\begin{align}\label{eqn:stochastic-parameter-choice-thm}
\delta &= \frac{\delta_s}{2304\Delta_f}\sqrt{\frac{\epsilon^3}{\rho}}
&\mathscr{T}_s&=\frac{8\ell}{\sqrt{\rho\epsilon}}\cdot\log\Big(\frac{\ell n}{\delta\sqrt{\rho\epsilon}}\Big),
&\iota &=10\log\Big(\frac{n\mathscr{T}_s^2}{\delta}\log\Big(\frac{\sqrt{n}}{\eta r_s}\Big)\Big),\\
r_s&=\frac{\delta}{480\rho n\mathscr{T}_s}\sqrt{\frac{\rho\epsilon}{\iota}},& m&=\frac{160(\ell+\tilde{\ell})}{\delta\sqrt{\rho\epsilon}}\sqrt{\mathscr{T}_s\iota},
&M&=\frac{16\ell\Delta_f}{\epsilon^{2}}
\end{align}
where $\Delta_f:=f(\vect{x}_0)-f^{*}$ and $f^{*}$ is the global minimum of $f$. Then, \algo{SGD+NC} satisfies that at least $1/4$ of the iterations $\vect{x}_{t}$ will be $\epsilon$-approximate second-order stationary points, using
\begin{align}
\tilde{O}\Big(\frac{(f(\vect{x}_{0})-f^{*})}{\epsilon^{4}}\cdot\log^2 n\Big)
\end{align}
iterations, with probability at least $1-\delta_s$.
\end{theorem}
\begin{proof}
Let the parameters be chosen according to \eqn{parameter-choice}, and set the total step number $T$ to be:
\begin{align}
T=\max\left\{\frac{8\ell(f(\vect{x_{0}})-f^{*})}{\epsilon^{2}},768(f(\vect{x_{0}})-f^{*})\cdot\sqrt{\frac{\rho}{\epsilon^3}}\right\}.
\end{align}
We will show that the following two claims hold simultaneously with probability $1-\delta_s$:
\begin{enumerate}
\item At most $T/4$ steps have gradients larger than $\epsilon$;
\item \algo{SNC-finding} can be called for at most $384\Delta_f\sqrt{\frac{\rho}{\epsilon^3}}$ times.
\end{enumerate}
Therefore, at least $T/4$ steps are $\epsilon$-approximate secondary stationary points. We prove the two claims separately.
\paragraph{Claim 1.}Suppose that within $T$ steps, we have more than $T/4$ steps with gradients larger than $\epsilon$. Then with probability $1-\delta_s/2$,
\begin{align}
f(\vect{x}_T)-f(\vect{x}_0)\leq -\frac{\eta}{8}\sum_{i=0}^{T-1}\|\nabla f(\vect{x}_i)\|^2+c\cdot\frac{\sigma^2}{M\ell}(T+\log(1/\delta_s))\leq f^{*}-f(\vect{x}_0),
\end{align}
contradiction.
\paragraph{Claim 2.}
We first assume that for each $\vect{x}_t$ we apply negative curvature finding (\algo{SNC-finding}), we can successfully obtain a unit vector $\hat{\vect{e}}$ with $\hat{\vect{e}}^{T}\mathcal{H}(\vect{x}_t)\hat{\vect{e}}\leq-\sqrt{\rho\epsilon}/4$, as long as $\lambda_{\min}(\mathcal{H}(\vect{x}_t))\leq-\sqrt{\rho\epsilon}$. The error probability of this assumption is provided later.

Under this assumption, \algo{SNC-finding} can be called for at most $384(f(\vect{x_{0}})-f^{*})\sqrt{\frac{\rho}{\epsilon^3}}\leq \frac{T}{2}$ times, for otherwise the function value decrease will be greater than $f(\vect{x_{0}})-f^{*}$, which is not possible. Then, the error probability that some calls to \algo{SNC-finding} fails is upper bounded by
\begin{align}
384(f(\vect{x_{0}})-f^{*})\sqrt{\frac{\rho}{\epsilon^3}}\cdot(3\delta)=\delta_s/2.
\end{align}

The number of iterations can be viewed as the sum of two parts, the number of iterations needed in large gradient scenario, denoted by $T_{1}$, and the number of iterations needed for negative curvature finding, denoted by $T_{2}$. With probability at least $1-\delta_s$,
\begin{align}
T_{1}=O(M\cdot T)=\tilde{O}\Big(\frac{(f(\vect{x}_{0})-f^{*})}{\epsilon^{4}}\Big).
\end{align}
As for $T_2$, with probability at least $1-\delta_s$, \algo{SNC-finding} is called for at most $384(f(\vect{x_{0}})-f^{*})\sqrt{\frac{\rho}{\epsilon^3}}$ times, and by \prop{SNC-finding} it takes $\tilde{O}\Big(\frac{\log^2 n}{\delta\sqrt{\rho\epsilon}}\Big)$ iterations each time. Hence,
\begin{align}
T_2=384(f(\vect{x_{0}})-f^{*})\sqrt{\frac{\rho}{\epsilon^3}}\cdot \tilde{O}\Big(\frac{\log^2 n}{\delta\sqrt{\rho\epsilon}}\Big)=\tilde{O}\Big(\frac{(f(\vect{x}_{0})-f^{*})}{\epsilon^{4}}\cdot\log^2 n\Big).
\end{align}
As a result, the total iteration number $T_1+T_2$ is
\begin{align}
\tilde{O}\Big(\frac{(f(\vect{x}_{0})-f^{*})}{\epsilon^{4}}\cdot\log^2 n\Big).
\end{align}
\end{proof}

%%%%%%%%%%%%%%%%%%%%%%%%%%%%%%%%%%%%%%%%%%%%%%%%%%%%%%

\section{More numerical experiments}\label{append:more-numerics}
In this section, we present more numerical experiment results that support our theoretical claims from a few different perspectives compared to \sec{numerical}. Specifically, considering that previous experiments all lies in a two-dimensional space, and theoretically our algorithms have a better dependence on the dimension of the problem $n$, it is reasonable to check the actual performance of our algorithm on high-dimensional test functions, which is presented in \append{dimension-dependence}. Then in \append{ANCGD&PAGD}, we introduce experiments on various landscapes that demonstrate the advantage of \algo{ANCGD} over PAGD \cite{jin2018accelerated}. Moreover, we compare the performance of our \algo{ANCGD} with the NEON$^+$ algorithm \cite{xu2017neon} on a few test functions in \append{ANCGD&NEON+}. To be more precise, we compare the negative curvature extracting part of NEON$^+$ with \algo{ANCGD} at saddle points in different types of nonconvex landscapes.

%=============================================

\subsection{Dimension dependence}\label{append:dimension-dependence}
Recall that $n$ is the dimension of the problem. We choose a test function $h(x) = \frac{1}{2}x^T \mathcal{H} x+\frac{1}{16}x_1^4$ where $\mathcal{H}$ is an $n$-by-$n$ diagonal matrix: $\mathcal{H} = \text{diag}(-\epsilon, 1, 1, ..., 1)$. The function $h(x)$ has a saddle point at the origin, and only one negative curvature direction. Throughout the experiment, we set $\epsilon = 1$. For the sake of comparison, the iteration numbers are chosen in a manner such that the statistics of \algo{NC-finding} and PGD in each category of the histogram are of similar magnitude.
\begin{figure}[htbp]
    \centering
    \includegraphics[width=0.8\textwidth]{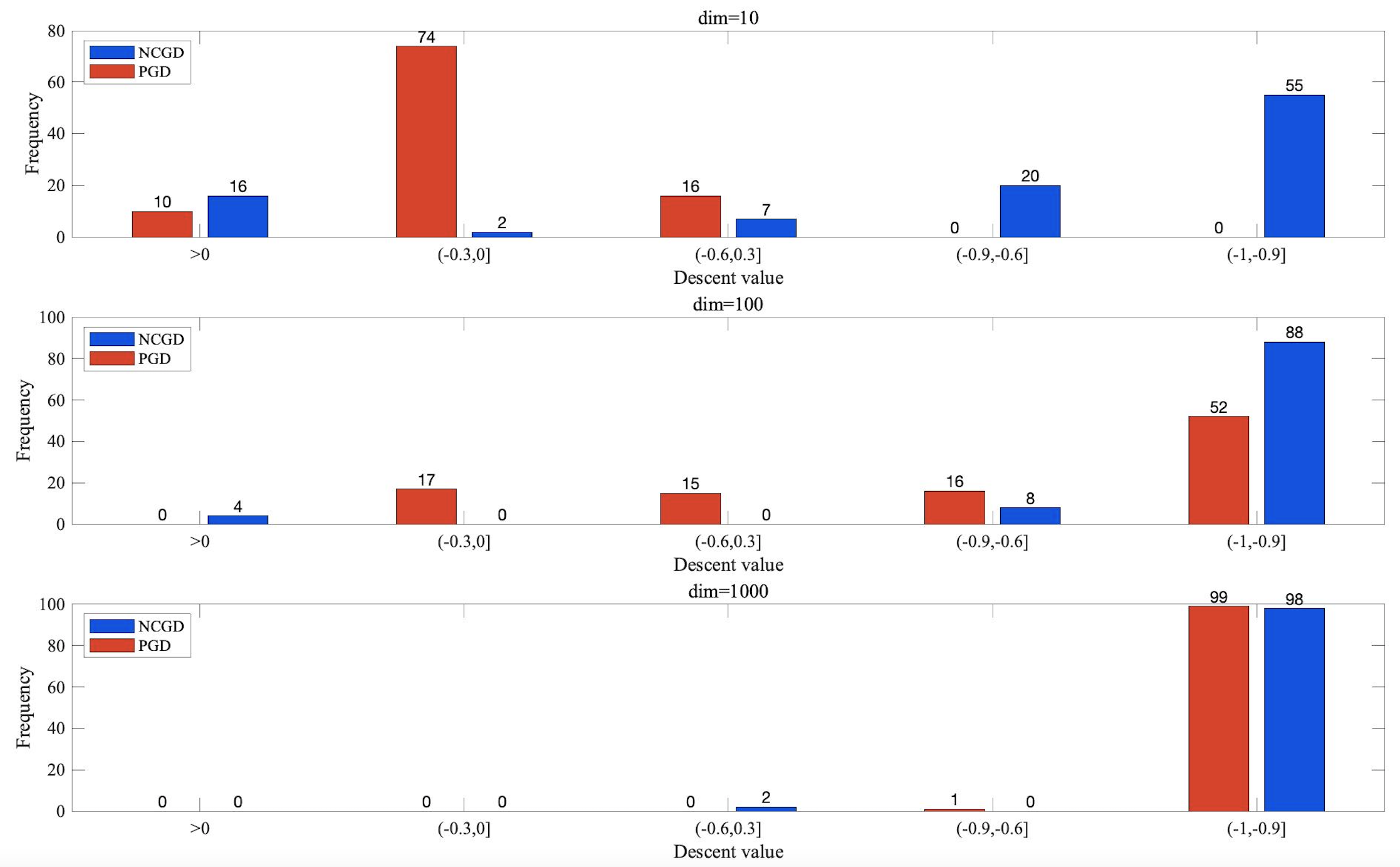}
    \caption{Dimension dependence of \algo{NC-finding} and PGD. We set $\epsilon = 0.01$, $r = 0.1$, $n = 10^p$ for $p = 1, 2, 3$. The iteration number of \algo{NC-finding} and PGD are separately set to be $30p$ and $20p^2+10$, and the sample size $M = 100$. As we can see, to maintain the same performance, the number of iterations in PGD grows faster than the number of iterations in \algo{NC-finding}.}
    \label{fig:dimension-dependence}
\end{figure}

%==============================================================================

\subsection{Comparison between \algo{ANCGD} and PAGD on various nonconvex landscapes}\label{append:ANCGD&PAGD}

\paragraph{Quartic-type test function}
Consider the test function $f(x_1,x_2)=\frac{1}{16}x_1^4-\frac{1}{2}x_1^2+\frac{9}{8}x_2^2$ with a saddle point at $(0,0)$. The advantage of \algo{ANCGD} is illustrated in \fig{AGD}.

\begin{figure}[H]
\centering\includegraphics[width = 12.2cm]{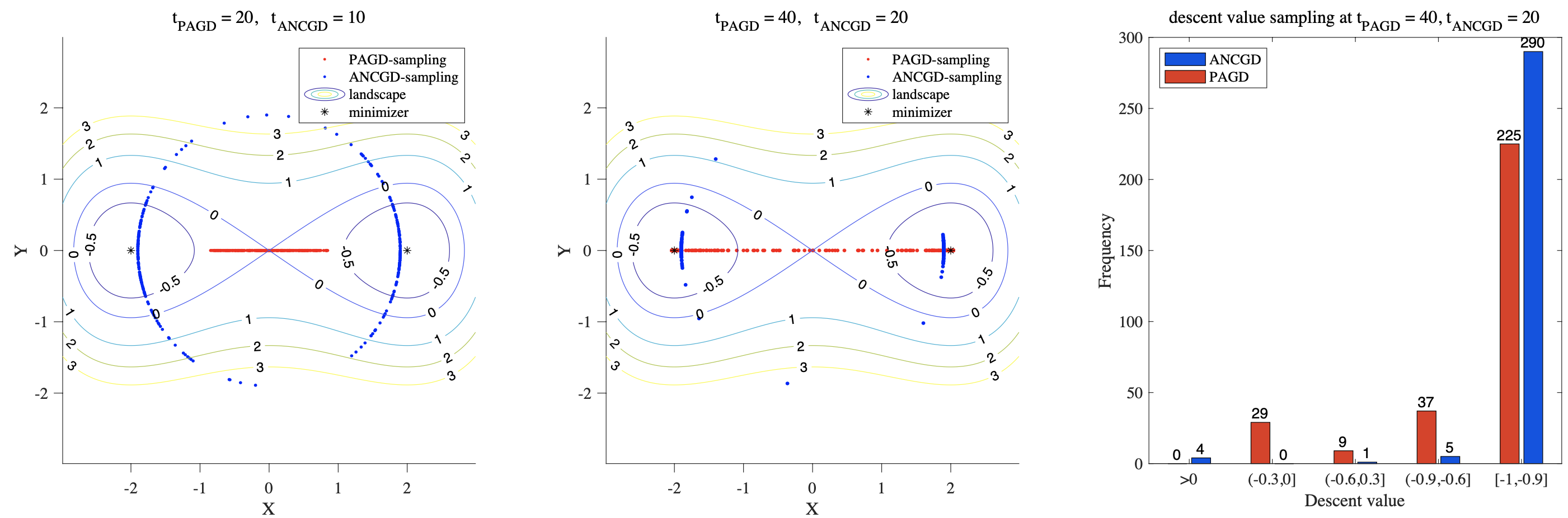}
\caption{Run \algo{ANCGD} and PAGD on landscape $f(x_1,x_2)=\frac{1}{16}x_1^4-\frac{1}{2}x_1^2+\frac{9}{8}x_2^2$. Parameters: $\eta =0.05$ (step length), $r = 0.08$ (ball radius in PAGD and parameter $r$ in \algo{ANCGD}), $M = 300$ (number of samplings).\\
\textbf{Left}: The contour of the landscape is placed on the background with labels being function values. Blue points represent samplings of \algo{ANCGD} at time step $t_{\text{ANCGD}}=10$ and $t_{\text{ANCGD}}=20$, and red points represent samplings of PAGD at time step $t_{\text{PAGD}}=20$ and $t_{\text{PAGD}}=40$. Similarly to \algo{NC-finding}, \algo{ANCGD} transforms an initial uniform-circle distribution into a distribution concentrating on two points indicating negative curvature, and these two figures represent intermediate states of this process. It converges faster than PAGD even when $t_{\text{ANCGD}}\ll t_{\text{PAGD}}$.\\
\textbf{Right}: A histogram of descent values obtained by \algo{ANCGD} and PAGD, respectively. Set $t_{\text{ANCGD}} = 20$ and $t_{\text{PAGD}} = 40$. Although we run two times of iterations in PAGD, there are still over $20\%$ of PAGD paths with function value decrease no greater than $0.9$, while this ratio for \algo{ANCGD} is less than $5\%$.}
\label{fig:AGD}
\end{figure}

\paragraph{Triangle-type test function.}
Consider the test function $f(x_1,x_2)=\frac{1}{2}\cos(\pi x_1)+\frac{1}{2}\Big(x_2+\frac{\cos(2\pi x_1)-1}{2}\Big)^2-\frac{1}{2}$ with a saddle point at $(0,0)$. The advantage of \algo{ANCGD} is illustrated in \fig{AGD-triangle}.

\begin{figure}[htbp]
\centering\includegraphics[width = 12.2cm]{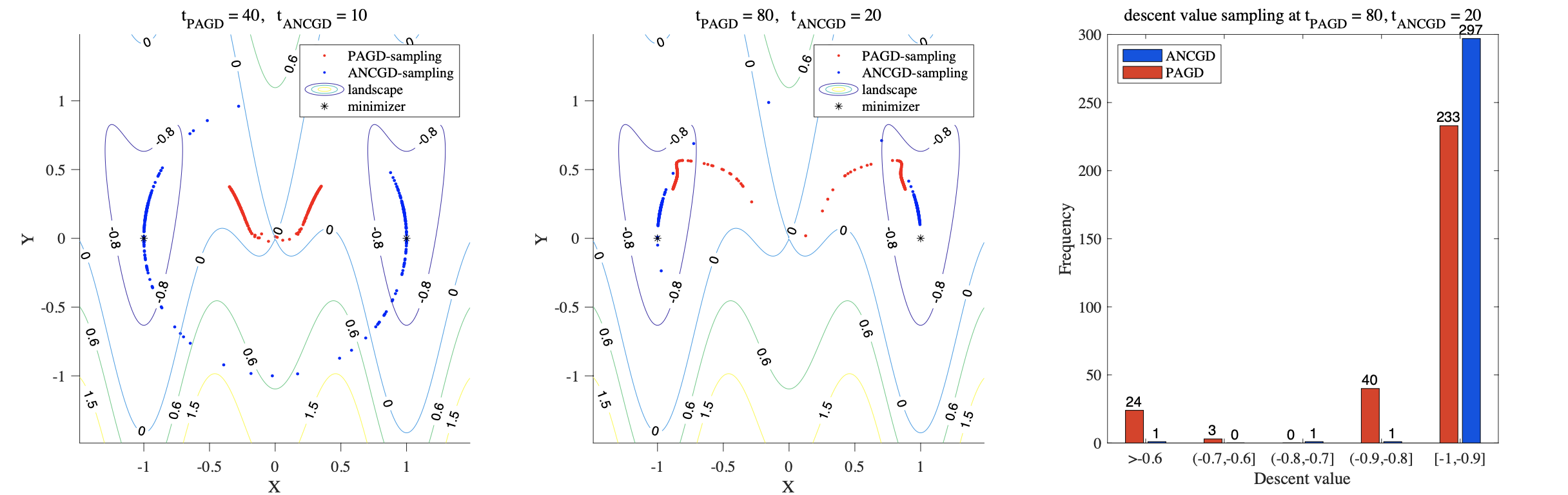}
\caption{Run \algo{ANCGD} and PAGD on landscape $f(x_1,x_2)=\frac{1}{2}\cos(\pi x_1)+\frac{1}{2}\Big(x_2+\frac{\cos(2\pi x_1)-1}{2}\Big)^2-\frac{1}{2}$. Parameters: $\eta =0.01$ (step length), $r = 0.1$ (ball radius in PAGD and parameter $r$ in \algo{ANCGD}), $M = 300$ (number of samplings).\\
\textbf{Left}: The contour of the landscape is placed on the background with labels being function values. Blue points represent samplings of \algo{ANCGD} at time step $t_{\text{ANCGD}}=10$ and $t_{\text{ANCGD}}=20$, and red points represent samplings of PAGD at time step $t_{\text{PAGD}}=40$ and $t_{\text{PAGD}}=80$. \algo{ANCGD} converges faster than PAGD even when $t_{\text{ANCGD}}\ll t_{\text{PAGD}}$.\\
\textbf{Right}: A histogram of descent values obtained by \algo{ANCGD} and PAGD, respectively. Set $t_{\text{ANCGD}} = 20$ and $t_{\text{PAGD}} = 80$. Although we run four times of iterations in PAGD, there are still over $20\%$ of gradient descent paths with function value decrease no greater than $0.9$, while this ratio for \algo{ANCGD} is less than $5\%$.}
\label{fig:AGD-triangle}
\end{figure}

\paragraph{Exponential-type test function.}
Consider the test function $f(x_1,x_2)=\frac{1}{1+e^{x_1^2}}+\frac{1}{2}\big(x_2-x_1^2e^{-x_1^2}\big)^2-1$ with a saddle point at $(0,0)$. The advantage of \algo{ANCGD} is illustrated in \fig{AGD-exp}.

\begin{figure}[htbp]
\centering\includegraphics[width = 12.2cm]{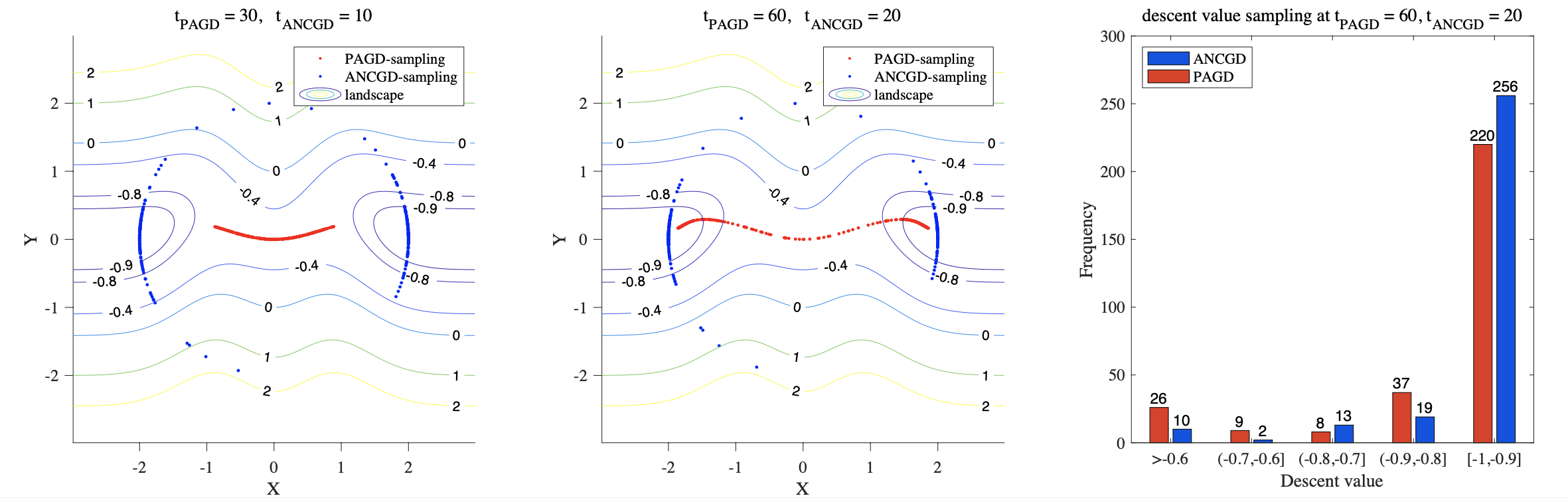}
\caption{Run \algo{ANCGD} and PAGD on landscape $f(x_1,x_2)=f(x_1,x_2)=\frac{1}{1+e^{x_1^2}}+\frac{1}{2}\big(x_2-x_1^2e^{-x_1^2}\big)^2-1$. Parameters: $\eta =0.03$ (step length), $r = 0.1$ (ball radius in PAGD and parameter $r$ in \algo{ANCGD}), $M = 300$ (number of samplings).\\
\textbf{Left}: The contour of the landscape is placed on the background with labels being function values. Blue points represent samplings of \algo{ANCGD} at time step $t_{\text{ANCGD}}=10$ and $t_{\text{ANCGD}}=20$, and red points represent samplings of PAGD at time step $t_{\text{PAGD}}=30$ and $t_{\text{PAGD}}=60$. \algo{ANCGD}converges faster than PAGD even when $t_{\text{ANCGD}}\ll t_{\text{PAGD}}$.\\
\textbf{Right}: A histogram of descent values obtained by \algo{ANCGD} and PAGD, respectively. Set $t_{\text{ANCGD}} = 20$ and $t_{\text{PAGD}} = 60$. Although we run three times of iterations in PAGD, its performance is still dominated by our \algo{ANCGD}.}
\label{fig:AGD-exp}
\end{figure}
Compared to the previous experiment on \algo{NC-finding} and PGD shown as \fig{quartic_descent} in \sec{numerical}, these experiments also demonstrate the faster convergence rates enjoyed by the general family of "momentum methods". Specifically, using fewer iterations, \algo{ANCGD} and PAGD achieve larger function value decreases separately compared to \algo{NC-finding} and PGD.

%=====================================================

\subsection{Comparison between \algo{ANCGD} and NEON$^{+}$ on various nonconvex landscapes}\label{append:ANCGD&NEON+}
\paragraph{Triangle-type test function.}
Consider the test function $f(x_1,x_2)=\frac{1}{2}\cos(\pi x_1)+\frac{1}{2}\Big(x_2+\frac{\cos(2\pi x_1)-1}{2}\Big)^2-\frac{1}{2}$ with a saddle point at $(0,0)$. The advantage of \algo{ANCGD} is illustrated in \fig{neon-cosine}.

\begin{figure}[htbp]
\centering\includegraphics[width = 12.2cm]{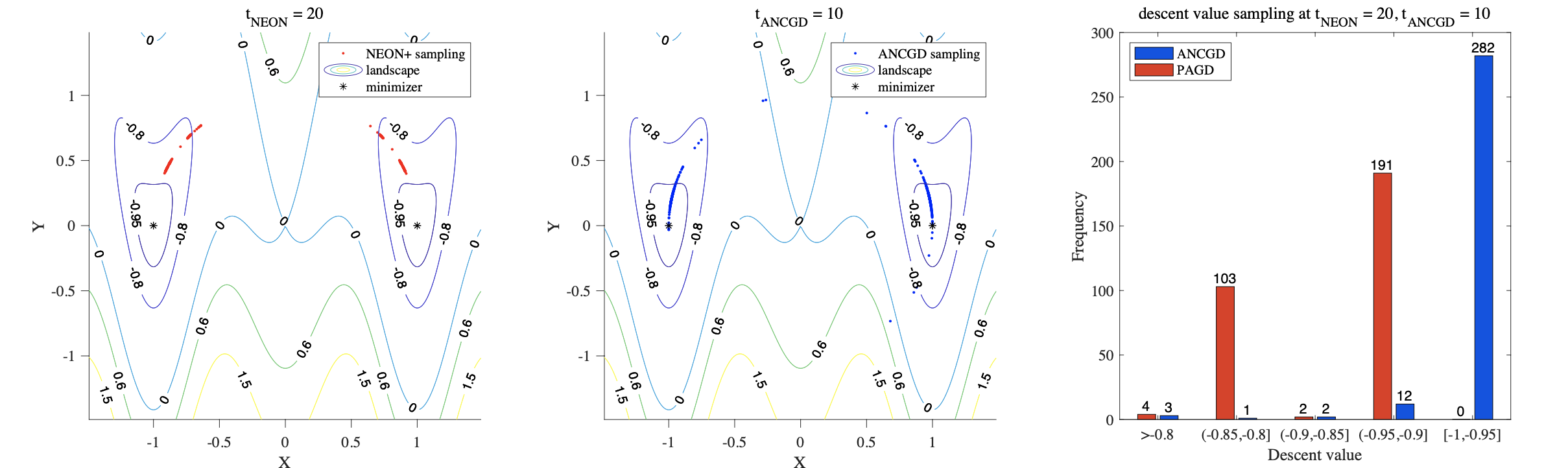}
\caption{Run \algo{ANCGD} and NEON$^+$ on landscape $f(x_1,x_2)=\frac{1}{2}\cos(\pi x_1)+\frac{1}{2}\big(x_2+\frac{\cos(2\pi x_1)-1}{2}\big)^2-\frac{1}{2}$. Parameters: $\eta =0.04$ (step length), $r = 0.1$ (ball radius in NEON$^+$ and parameter $r$ in \algo{ANCGD}), $M = 300$ (number of samplings).\\
\textbf{Left}: The contour of the landscape is placed on the background with labels being function values. Red points represent samplings of NEON$^+$ at time step $t_{\text{NEON}}=20$, and blue points represent samplings of \algo{ANCGD} at time step $t_{\text{ANCGD}}=10$. \algo{ANCGD} and the negative curvature extracting part of NEON$^+$ both transform an initial uniform-circle distribution into a distribution concentrating on two points indicating negative curvature. Note that \algo{ANCGD} converges faster than NEON$^+$ even when $t_{\text{ANCGD}}\ll t_{\text{NEON}}$.\\
\textbf{Right}: A histogram of descent values obtained by \algo{ANCGD} and NEON$^+$, respectively. Set $t_{\text{ANCGD}} = 10$ and $t_{\text{NEON}} = 20$. Although we run two times of iterations in NEON$^+$, none of NEON$^+$ paths has function value decrease greater than $0.95$, while this ratio for \algo{ANCGD} is larger than $90\%$.}
\label{fig:neon-cosine}
\end{figure}

\paragraph{Exponential-type test function.}
Consider the test function $f(x_1,x_2)=\frac{1}{1+e^{x_1^2}}+\frac{1}{2}\big(x_2-x_1^2e^{-x_1^2}\big)^2-1$ with a saddle point at $(0,0)$. The advantage of \algo{ANCGD} is illustrated in \fig{neon-exp}

\begin{figure}[htbp]
\centering\includegraphics[width = 12.2cm]{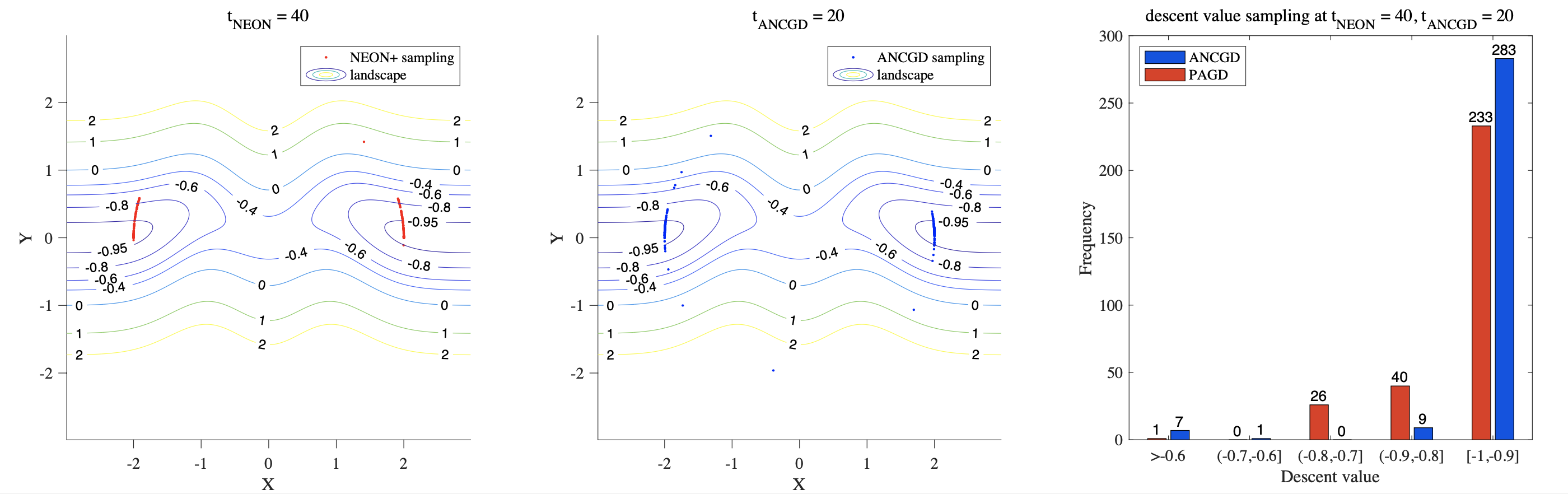}
\caption{Run \algo{ANCGD} and NEON$^+$ on landscape $f(x_1,x_2)=\frac{1}{1+e^{x_1^2}}+\frac{1}{2}\big(x_2-x_1^2e^{-x_1^2}\big)^2-1$. Parameters: $\eta =0.03$ (step length), $r = 0.1$ (ball radius in NEON$^+$ and parameter $r$ in \algo{ANCGD}), $M = 300$ (number of samplings).\\
\textbf{Left}: The contour of the landscape is placed on the background with labels being function values. Red points represent samplings of NEON$^+$ at time step $t_{\text{NEON}}=40$, and blue points represent samplings of \algo{ANCGD} at time step $t_{\text{ANCGD}}=20$. \algo{ANCGD} converges faster than NEON$^+$ even when $t_{\text{ANCGD}}\ll t_{\text{NEON}}$.\\
\textbf{Right}: A histogram of descent values obtained by \algo{ANCGD} and NEON$^+$, respectively. Set $t_{\text{ANCGD}} = 20$ and $t_{\text{NEON}} = 40$. Although we run two times of iterations in NEON$^+$, there are still over $20\%$ of NEON$^+$ paths with function value decrease no greater than $0.9$, while this ratio for \algo{ANCGD} is less than $10\%$.}
\label{fig:neon-exp}
\end{figure}

Compared to the previous experiments on \algo{ANCGD} and PAGD in \append{ANCGD&PAGD}, these two experiments also reveal the faster convergence rate of both NEON$^+$ and \algo{ANCGD} against PAGD \cite{jin2018accelerated} at small gradient regions.

%%%%%%%%%%%%%%%%%%%%%%%%%%%%%%%%%%%%%%%%%%%%%%%%%%%%%%%%%%

\end{document}